\documentclass{article}

 \usepackage{amsmath,amssymb,amsthm,mathrsfs, tikz, graphicx}
 \usepackage{setspace, varioref, hyperref}
 \usepackage{enumitem}
\newtheorem{theo}{Theorem}
\labelformat{theo}{Theorem~#1} 
\newtheorem{lemme}{Lemma}
\labelformat{lemme}{Lemma~#1} 
\newtheorem{defi}{Definition}
\labelformat{defi}{definition~#1} 
\newtheorem{hyp}{Assumption}
\labelformat{hyp}{Assumption~#1} 
\newtheorem{coro}{Corollary}
\labelformat{coro}{Corollary~#1}
\labelformat{section}{Section~#1}
\labelformat{subsection}{Subsection~#1}
\labelformat{figure}{Figure~#1}
\numberwithin{lemme}{section}
\numberwithin{theo}{section}
\numberwithin{defi}{section}
\numberwithin{coro}{section}
\numberwithin{equation}{section}
\newtheorem{Rem}{Remark}
\labelformat{Rem}{Remark~#1} 
\newcommand{\RR}{\mathbb{R}}
\newcommand{\NN}{\mathbb{N}}

\newcommand{\Ss}{\mathbb{S}^{d-1}}
\newcommand{\vp}{\varphi}

\newcommand{\ep}{\varepsilon}
\newcommand{\Ep}{^{\ep}}
\newcommand{\xep}{\xi^{\varepsilon}}
\newcommand{\Oep}{\Omega_{\varepsilon}}
\newcommand{\Eep}{E^{\varepsilon}}
\newcommand{\Nep}{N^{\varepsilon}}
\newcommand{\Lxsp}{L_{\xi}(\sigma, \rho)}
\newcommand{\Ob}{\overline{\Omega}}
\newcommand{\Qb}{\overline{Q}}

\newcommand{\Lb}{\overline{L}}
\newcommand{\rhob}{\overline{\rho}}

\newcommand{\Mc}{\mathcal{M}}

\newcommand{\Pc}{\mathcal{P}}
\newcommand{\Lc}{\mathcal{L}}
\newcommand{\Pf}{\mathfrak{P}}

\newcommand{\Mfs}{\mathfrak{M}_{\sigma}}
\newcommand{\Gc}{\mathcal{G}}

\newcommand{\Qc}{\mathcal{Q}}
\newcommand{\Xc}{\mathcal{X}}
\newcommand{\Sc}{\mathcal{S}}
\newcommand{\Lct}{\widetilde{\mathcal{L}}}
\newcommand{\Qt}{\widetilde{Q}}
\newcommand{\Xt}{\widetilde{X}}

\newcommand{\Sct}{\widetilde{\Sc}}

\newcommand{\Ltxs}{\widetilde{L}_{\xi}(\sigma)}
\newcommand{\xig}{\boldsymbol{\xi}}
\newcommand{\nut}{\widetilde{\nu}}
\newcommand{\nutl}{\nu_t \otimes \lambda}

\newcommand{\ObOb}{\Ob \times \Ob}

\newcommand{\OS}{\Omega \times \Ss} 
\newcommand{\ObS}{\Ob \times \Ss} 
\newcommand{\IO}{\int_{\Omega}} 
\newcommand{\IOb}{\int_{\Ob}} 
\newcommand{\IOO}{\int_{\Omega \times \Omega}} 
\newcommand{\IObOb}{\int_{\ObOb}}
 
\newcommand{\IOS}{\int_{\Omega \times \Ss}} 
\newcommand{\IL}{\int_{\Lc}} 
\newcommand{\ISc}{\int_{\Sc}}

\newcommand{\cbx}{\overline{c}_{\xi}} 
\newcommand{\bal}{\begin{align*}}
\newcommand{\eal}{\end{align*}}
\newcommand{\ble}{\begin{lemme}}
\newcommand{\ele}{\end{lemme}}
\newcommand{\bpr}{\begin{proof}}
\newcommand{\epr}{\end{proof}}
\newcommand{\xxe}{\xi \left ( x, \frac{e}{|e|} \right )}

\newcommand{\beq}{\begin{equation}}
\newcommand{\eeq}{\end{equation}}

\newcommand{\rhonu}{\rho^{\nu}}

\begin{document}

\title{Wardrop equilibria : rigorous derivation of continuous limits from general networks models}

\author{Rom\'eo Hatchi \thanks{\scriptsize CEREMADE, UMR CNRS 7534, Universit\'e Paris IX
Dauphine, Pl. de Lattre de Tassigny, 75775 Paris Cedex 16, FRANCE
\texttt{hatchi@ceremade.dauphine.fr}.}}

\maketitle

\begin{abstract}
The concept of Wardrop equilibrium plays an important role in congested traffic problems since its introduction in the early $50$'s. As shown in \cite{baillon2012discrete}, when we work in two-dimensional cartesian and increasingly dense networks, passing to the limit by $\Gamma$-convergence, we obtain continuous minimization problems posed on measures on curves. Here we study the case of general networks in $\RR^d$ which become very dense. We use the notion of generalized curves and extend the results of the cartesian model.
\end{abstract}


\section{Introduction}

Modeling congested traffic is a field of research that has developed especially since the early $50$'s and the introduction of Wardrop equilibrium \cite{wardrop1952road}. Its popularity is due to many applications to road traffic and more recently to communication networks. In our finite networks model, we represent the congestion effects by the fact that the traveling time of each arc is a nondecreasing function of the flow on this arc. The concept of Wardrop equilibrium simply says that all used roads between two given points have the same cost and this cost is minimal. So we assume a rational behavior by users. A Wardrop equilibrium is a flow configuration that satisfies mass conservation conditions and positivity constraints. A few years after Wardrop defined his equilibrium notion, Beckmann, McGuire and Winsten \cite{beckmann1956studies} observed that Wardrop equilibrium can be formulated in terms of solutions of a convex optimization problem. However this variational characterization uses the whole path flow configuration. It becomes very costly when working in increasingly dense networks. We may often prefer to study the dual problem which is less untractable. But finding an optimal solution remains a hard problem because of the presence of a nonsmooth and nonlocal term.  As we study a sequence of discrete networks increasingly dense, it is natural to ask under what conditions we can pass to a continuous limit which would simplify the problem.

The purpose of this paper is to rigorously justify passing to the limit thanks to the theory of $\Gamma$-convergence and then to find a continuous analogue of Wardrop equilibrium. We will strongly rely on two articles \cite{carlier2008optimal} and \cite{baillon2012discrete}. The first establishes some first results on traffic congestion. The second studies the case of a two-dimensional cartesian grid with small arc length $\ep$. Here we will consider general networks in $\RR^d$ with small arc length of order $\ep$. It is a substantial improvement. We will show the $\Gamma$-convergence of the functionals in the dual problem as $\ep$ goes to $0$. We will obtain an optimization problem over metrics variables. The proof of the $\Gamma$-convergence is constructed in the same manner as in \cite{baillon2012discrete}. But two major difficulties here appear. Indeed in the case of the grid in $\RR^2$, there are only four possible constant directions (that are $((1,0),(0,1),(-1,0),(0,-1))$) so that for all speed $z \in \RR^2$, there exists an unique decomposition of $z$ in the family of these directions, with positive coefficients. In the general case, directions are not necessarily constant and we have no uniqueness of the decomposition. To understand how to overcome these obstacles, we can first look at the case of regular hexagonal networks. There are six constant directions (that are $exp(i(\pi/6+k \pi/3)), k=0, \dots, 5$) but we lose uniqueness. Then, we can study the case of a two-dimensional general network in which directions can vary and arcs lengths are not constant. The generalization from $\RR^2$ to $\RR^d$ (where $d$ is any integer $\geq 2$) is simpler. Of course, it is necessary to make some structural assumptions on the networks to have the $\Gamma$-convergence. These hypotheses are satisfied for instance in the cases of the isotropic model in \cite{carlier2008optimal} and the cartesian one in \cite{baillon2012discrete}. 

The limit problem (in the general case) is the dual of a continuous problem posed on a set of probability measures over generalized curves of the form $(\sigma, \rho)$ where $\sigma$ is a path and $\rho$ is a positive decomposition of $\dot{\sigma}$ in the family of the directions. This takes the anisotropy of the network into account. We will then remark that we can define a continuous Wardrop equilibrium through the optimality conditions for the continuous model. To establish that the limit problem, has solutions we work with a relaxation of this problem through the Young's measures and we extend results in \cite{carlier2008optimal}. Indeed we cannot directly apply Prokhorov's theorem to the set of generalized curves since weak-$L^1$ is not contained in a Polish space. First we are interested in the short-term problem, that is, we have a transport plan that gives the amount of mass sent from each source to each destination. We may then generalize these results to the long-term variant in which only the marginals (that are the distributions of supply and demand) are known. This case is interesting since as developed in \cite{brasco2010congested, brasco2013congested, hatchi2015wardrop}, it amounts to solve a degenerate elliptic PDE. But it will not be developed here. 

The plan of the paper is as follows: Section $2$ is devoted to a preliminary description of the discrete model with its notations, definition of Wardrop equilibrium and its variational characterization. In Section $3$, we explain the assumptions made and we identify the limit functional. We then state the $\Gamma$-convergence result. The proof is given in Section $4$. Then, in Section $5$, we formulate the optimality conditions  for the limit problem that lead to a continuous Wardrop equilibrium. Finally, in Section $6$, we adapt the previous results to the long-term problem.

\section{The discrete model}

\subsection{Notations and definition of Wardrop equilibria}

Let $d \in \NN, d \geq 2$ and $\Omega$ a bounded domain of $\RR^d$ with a smooth boundary and $\ep > 0$. We consider a sequence of discrete networks $\Oep = (N^{\ep}, E^{\ep})$ whose characteristic length is $\ep$, where $N^{\ep}$ is the set of nodes in $\Omega_{\ep}$ and $E^{\ep}$ the (finite) set of pairs $(x,e)$ with $x \in \Nep$ and $e \in \RR^d$ such that the segment $[x, x+e]$ is included in $\Omega$. We will simply identify arcs to pairs $(x,e)$. We impose $|E^{\ep}| = \max \{|e|, \text{ there exists }x \text{ such that }(x,e) \in E^{\ep} \} = \ep$. We may assume that two arcs can not cross. The orientation is important since the arcs $(x,e)$ and $(x+e, -e)$ really represent two distinct arcs. Now let us give some definitions and notations.

\textbf{Traveling times and congestion:} We denote the mass commuting on arc $(x,e)$ by $m^{\ep}(x,e)$ and the traveling time of arc $(x,e)$ by $t^{\ep}(x,e)$. We represent congestion by the following relation between traveling time and mass for every arc $(x,e)$:
\beq \label{2.1}
t^{\ep}(x,e)=g^{\ep}(x,e, m^{\ep}(x,e))
\eeq
where for every $\ep$, $g\Ep$ is a given positive function that depends on the arc itself but also on the mass $m^{\ep}(x,e)$ that commutes on the arc $(x,e)$ in a nondecreasing way: this is congestion. We will denote the set of all arc-masses $m\Ep(x,e)$ by $\mathbf{m\Ep}$. Orientation of networks here is essential: considering two neighboring nodes $x$ and $x'$ with $(x,x'-x)$ and $(x', x-x') \in \Eep$, the time to go from $x$ to $x'$ only depends on the mass $m\Ep(x, x'-x)$ that uses the arc $(x, x'-x)$ whereas the time to go from $x'$ to $x$ only depends on the mass $m\Ep(x', x-x')$.

\textbf{Transport plan:} A transport plan is a given function $\gamma^{\ep} : N^{\ep} \times N^{\ep} \mapsto \mathbb{R}_+$. That is a collection of nonnegative masses where for each pair $(x,y) \in N\Ep \times N\Ep$ (viewed as a source/destination pair), $\gamma^{\ep}(x,y)$ is the mass that has to be sent from the source $x$ to the target $y$. 

\textbf{Paths:} A path is a finite collection of successive nodes. We therefore represent a path $\sigma$ by writing $\sigma= \left (x_0, \dots, x_{L(\sigma)} \right )$ with $\sigma(k)=x_k \in N^{\ep}$ and $(\sigma(k), \sigma(k+1) - \sigma(k)) \in E^{\ep}$ for $k=0, \dots, L(\sigma)-1$. $\sigma(0)$ is the origin of $\sigma$ and $\sigma(L(\sigma))$ is the terminal point of $\sigma$. The length of $\sigma$ is 
\[
\sum_{k=0}^{L(\sigma)-1} | x_{k+1}-x_k|.
\]
We say that $(x,e) \subset \sigma$ if there exists $k \in \{1, \dots, L(\sigma)-1 \} $ such that $\sigma(k) = x$ and $e = \sigma(k+1) - \sigma(k)$. Since the time to travel on each arc is positive, we can impose $\sigma$ has no loop. We will denote the (finite) set of loop-free paths by $C\Ep$, that may be partitioned as
\[
C\Ep = \bigcup_{(x,y) \in N\Ep \times N\Ep} C_{x,y}\Ep,
\]
where $C_{x,y}\Ep$ is the set of loop-free paths starting from the origin $x$ and stopping at the terminal point $y$. The mass commuting on the path $\sigma \in C\Ep$ will be denoted $w\Ep(\sigma)$. The collection of all path-masses $w\Ep(\sigma)$ will be denoted $\mathbf{w\Ep}$. Given arc-masses $\mathbf{m\Ep}$, the travel time of a path $\sigma \in C\Ep$ is given by:  
\[
\tau_{\mathbf{m^{\ep}}}^{\ep}(\sigma)= \sum_{(x,e) \subset \sigma} g^{\ep} (x,e,m^{\ep}(x,e)).
\]

\textbf{Equilibria:} In short, in this model, the data are the masses $\gamma\Ep(x,y)$ and the congestion functions $g\Ep$. The unknowns are the arc-masses $m\Ep(x,e)$ and path-masses $w\Ep(\sigma)$. We wish to define some equilibrium requirements on these unknowns. First, they should be nonnegative. Moreover, we have the following conditions that relate arc-masses, path-masses and the data $\gamma\Ep$ :
\begin{equation} \label{cons1}
\gamma\Ep (x, y) := \sum_{\sigma \in C_{x, y}\Ep} w\Ep(\sigma), \: \forall (x, y) \in \Nep \times \Nep
\end{equation}
and
\begin{equation} \label{cons2}
m\Ep(x, e) = \sum_{\sigma \in C\Ep: (x,e) \subset \sigma} w\Ep(\sigma), \forall (x,e) \in \Eep. 
\end{equation}
Both express mass conservation. We finally require that only the shortest paths (taking into account the congestion created by arc and path-masses) should actually be used. This is the concept of Wardrop equilibrium that is defined precisely as follows:

\begin{defi} \label{defW}
A Wardrop equilibrium is a configuration of nonnegative arc-masses $\mathbf{m}\Ep : (x, e) \rightarrow (m\Ep(x,e))$ and of nonnegative path-masses $\mathbf{w}\Ep : \sigma \rightarrow w\Ep(\sigma)$, satisfying the mass conservation conditions \eqref{cons1} and \eqref{cons2} and such that for every $(x, y) \in  \Nep \times \Nep$ and every $\sigma \in C_{x,y}\Ep $, if $w\Ep(\sigma) > 0$ then 
\[
\tau_{\mathbf{m^{\ep}}}^{\ep}(\sigma) \leq \tau_{\mathbf{m^{\ep}}}^{\ep}(\sigma'), \forall \sigma' \in C_{x,y}\Ep.
\]
\end{defi}  

\subsection{Variational characterizations of equilibria}

Soon after the work of Wardrop, Beckmann, McGuire and Winsten \cite{beckmann1956studies} discovered that Wardrop equilibria can be obtained as minimizers of a convex optimization problem: 
\begin{theo} A flow configuration $(\mathbf{w}^{\ep}, \mathbf{m}^{\ep})$ is a Wardrop equilibrium if and only if it minimizes
\begin{equation} \label{P1}
\sum_{(x,e) \in E^{\ep}} G^{\ep} (x, e, m^{\ep}(x,e)) \text{ where } G^{\ep}(x, e, m) := \int_0^m g^{\ep}(x, e, \alpha) d \alpha 
\end{equation}
subject to nonnegativity constraints and the mass conservation conditions \eqref{cons1}-\eqref{cons2}. 
\end{theo}


Since the functions $g\Ep$ are nondecreasing with respect to the last variable, the problem \eqref{P1} is convex so we can easily obtain existence results and numerical schemes. Unfortunately, this problem becomes quickly costly whenever the network is very dense, since it requires to enumerate all paths flows $w\Ep(\sigma)$. For this reason, we can not use this characterization to study realistic congested networks. An alternative consists in working with the dual formulation which is
\begin{equation} \label{D1}
 \inf_{t^{\ep} \in \mathbb{R}_+^{\#E^{\ep}}} \sum_{(x,e) \in E^{\ep}} H^{\ep} (x, e, t^{\ep} (x,e)) - \sum_{(x,y) \in {N^{\ep}}^2} \gamma^{\ep}(x,y) T_{\mathbf{t^{\ep}}}^{\ep}(x,y),
 \end{equation}
 where $\mathbf{t\Ep} \in \RR_+^{\#\Eep}$ should be understood as $\mathbf{t\Ep} = (t\Ep(x,e))_{(x,e) \in \Eep}$, $H^{\ep}(x,e, .) := (G\Ep(x,e,.))^* $ is the Legendre transform of $G^{\ep}(x,e, .)$ that is 
 \[
 H\Ep(x,e,t) := \sup_{m \geq 0} \{ mt - G\Ep(x,e,m) \}, \: \forall t \in \RR_+ 
 \]
 and $T\Ep_{\mathbf{t\Ep}}$ is the minimal length functional: 
\[
 T_{\mathbf{t^{\ep}}}^{\ep}(x,y)= \min_{\sigma \in C_{x,y}^{\ep}} \sum_{(x,e) \subset \sigma} t^{\ep} (x,e).
\]
The complexity of \eqref{D1} seems better since we only have $\# \Eep = O(\ep^{-d})$ nodes variables. However an important disadvantage appears in the dual formulation. The term $T_{\mathbf{t^{\ep}}}^{\ep}$ is nonsmooth, nonlocal and we might have difficulties to optimize that. Nevertheless we will see that we may pass to a continuous limit which will simplify the structure since we can then use the Hamilton-Jacobi theory. 
 
 \section{The $\Gamma$-convergence result}
 \subsection{Assumptions}
 
We obviously have to make some structural assumptions on the $\ep$-dependence of the networks and the data to be able to pass to a continuous limit in the Wardrop equilibrium problem. To understand all these assumptions, we will illustrate with some examples. Here, we will consider the cases of regular decomposition (cartesian, triangular and hexagonal, see \ref{hexag}) for $d=2$. In these models, all the arcs in $\Oep$ have the same length that is $\ep$. We will introduce some notations and to refer to a specific example, we will simply add the letters c (for the cartesian case), t (for the triangular one) and h (for the hexagonal one). 

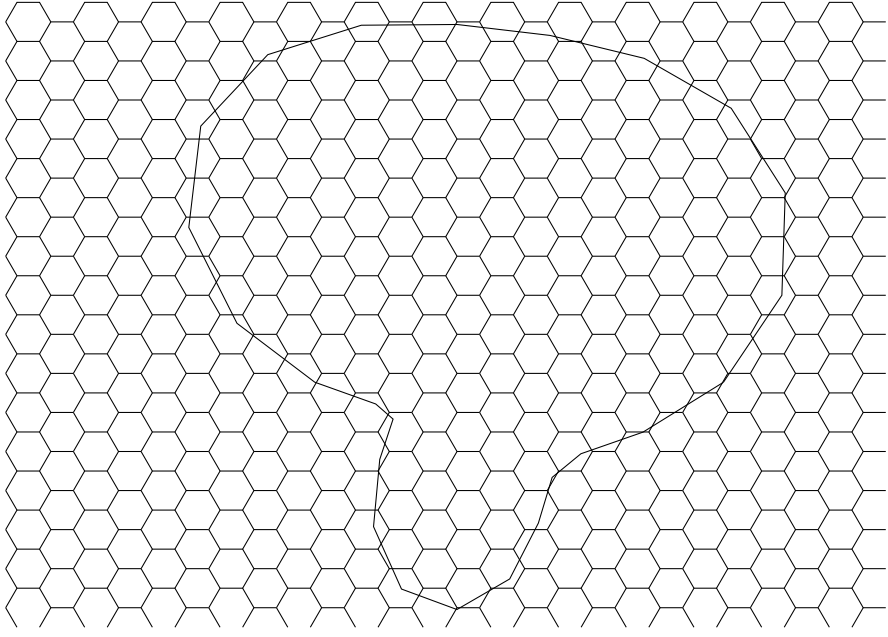
\begin{figure}
\begin{tikzpicture}
      \foreach \i in {0,3,...,36}{
          \foreach \j in {0,...,15}{
            \begin{scope}[xshift=\i*0.3 cm,yshift={\j*0.5196152 cm}]
               \draw (0:0) -- ++ (60:0.3cm)-- ++ (0:0.3cm)--++(-60:0.3cm)--++(0:0.3cm);
               \draw (0:0)--  ++ (-60:0.3cm) ++ (0:0.3cm) --++(60:0.3cm);
            \end{scope}
            }
         }
        \draw [domain=0:2] plot({6+1.4*cos(deg(\x*pi))*(2.36+0.38*cos(deg(\x*pi))+1.16*sin(deg(\x*pi))-0.42*cos(deg(2*\x*pi))+0.21*cos(deg(3*\x*pi))+0.54*sin(deg(3*\x*pi)))},{3+1.4*sin(deg(\x*pi))*(2.36+0.38*cos(deg(\x*pi))+1.16*sin(deg(\x*pi))-0.42*cos(deg(2*\x*pi))+0.21*cos(deg(3*\x*pi))+0.54*sin(deg(3*\x*pi)))});  
 \end{tikzpicture}
 \caption{An example of domain in $2d$-hexagonal model} 
 \label{hexag}
    \end{figure}

The first assumption concerns the length of the arcs in the networks.

\begin{hyp} \label{hyarc}
There exists a constant $C >0$ such that for every $\ep >0, (x,e) \in \Eep,$ we have $C \ep \leq |e| \leq \ep$. 
\end{hyp}

More generally we denote by $C$ a generic constant that does not depend on the scale parameter $\ep$. 

The following assumption is on the discrete network $\Oep$, $\ep > 0$. Roughly speaking, the arcs of $\Oep$ define a bounded polyhedron (still denoted by $\Oep$ by abuse of notations) that is an approximation of $\Ob$. $\Oep$ is the union of cells - or polytopes - $(V\Ep_i)$. Each $V\Ep_i$ is itself the union of subcells - or facets - $(F\Ep_{i,k})$. More precisely, we have:

 \begin{hyp} \label{hyre}
Up to a subsequence, $\Oep \subset \Omega_{\ep'}$ for $\ep > \ep' > 0$ and $\Ob = \bigcup_{\ep > 0} \Oep$. For $\ep > 0$, we have 
\[
\Oep = \bigcup_{i \in I\Ep} V\Ep_i, \: I\Ep \text{ is finite.}
\] 
There exists $S \in \NN$ and $s_i \leq S$ such that for every $\ep > 0$ and $i \in I\Ep$, $V\Ep_i = \text{Conv}(x_{i,1}\Ep, \dots, x_{i,s_i}\Ep)$ where for $j=1, \dots, s_i - 1$, $x\Ep_{i,j}$ is a neighbor of $x\Ep_{i,j+1}$ and $x\Ep_{i, s_i}$ is a neighbor of $x\Ep_{i,1}$ in $\Oep$. Its interior $\mathring{V}\Ep_i$ contains no arc $\in \Eep$. For $i \neq j, V\Ep_i \cap V\Ep_j = \emptyset$ or exactly a facet (of dimension $\leq d-1$). Let us denote $X\Ep_i$ the isobarycenter of all nodes $x\Ep_{i,k}$ contained in $V\Ep_i$. For $i \in I\Ep$, we have 
\[
V\Ep_i = \bigcup_{j \in I\Ep_i} F\Ep_{i,j} \text{ with } I\Ep_i \text{ finite}. 
\]
For $j \in I\Ep_i$, $F\Ep_{i,j} = \text{Conv} (X\Ep_i, x\Ep_{i,j_1}, \dots, x\Ep_{i, j_d})$ with these $(d+1)$ points that are affinely independent. For every $k \neq l$, $F\Ep_{i,k} \cap F\Ep_{i,l}$ is a facet of dimension $\leq d-1$, containing $X\Ep_i$. There exists a constant $C > 0$ independent of $\ep$ such that for every $i \in I\Ep$ and $j \in I\Ep_i$, the volume of $V\Ep_i$ and $F\Ep_{i,j}$ satisfy
\[
\frac{1}{C} \ep^d < |F\Ep_{i,j}| \leq |V\Ep_i| < C \ep^d.
\]

 \end{hyp}
 
 The estimate on the volumes implies another estimate: for every $i \in I\Ep$ and $k=1, \dots, s_i$, we have
 \beq \label{estiso}
 \frac{1}{C} \ep < \text{dist}(X\Ep_i, x\Ep_{i,k}) < C \ep
 \eeq
(the constant $C$ is not necessarily the same).  This hypothesis will in particular allow us to make possible a discretization of Morrey's theorem and then to prove \ref{dismor}. It is a non trivial extension of what is happening in dimension $2$. \ref{illhy2} shows an illustration of \ref{hyre} in the cartesian case. 

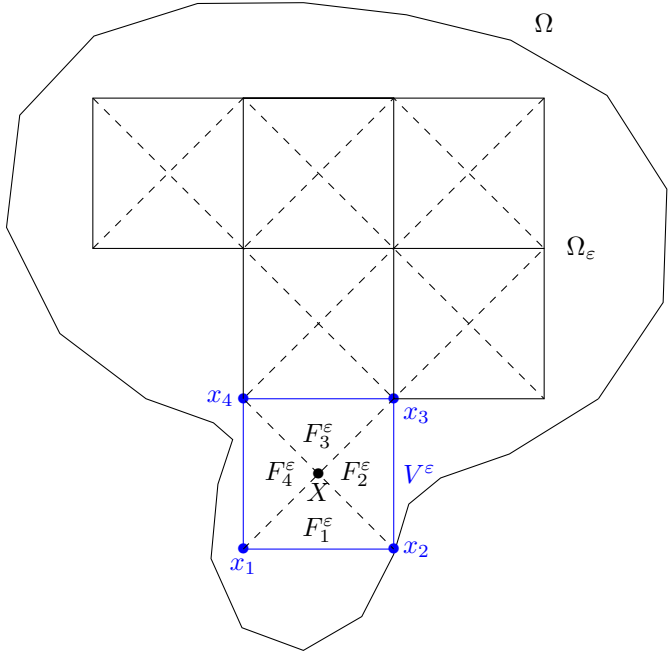
\begin{figure}[!ht]
\begin{tikzpicture}
        \draw [domain=0:2] plot({0.8+1.55*cos(deg(\x*pi))*(2.36+0.38*cos(deg(\x*pi))+1.16*sin(deg(\x*pi))-0.42*cos(deg(2*\x*pi))+0.21*cos(deg(3*\x*pi))+0.54*sin(deg(3*\x*pi)))},{2+1.55*sin(deg(\x*pi))*(2.36+0.38*cos(deg(\x*pi))+1.16*sin(deg(\x*pi))-0.42*cos(deg(2*\x*pi))+0.21*cos(deg(3*\x*pi))+0.54*sin(deg(3*\x*pi)))});  
        \draw [color = blue] (0,0) node[below] {$x_1$} node {$\bullet$} -- (2,0) node[right] {$x_2$} node{$\bullet$} -- (2,2) node[below right] {$x_3$} node {$\bullet$} -- (0,2) node[left] {$x_4$} node {$\bullet$} -- (0,0); 
       \draw [color = blue] (2,1) node[right] {$V\Ep$}; 
       \draw (1,0.25) node {$F\Ep_1$} ++ (0.5, 0.75) node {$F\Ep_2$} ++ (-1,0) node {$F\Ep_4$} ++ (0.5, 0.5) node {$F\Ep_3$};
        \draw (1,1) node[below, fill=white] {$X$} node {$\bullet$};
        \draw (0,2) -- (0,6) -- (2,6) -- (2,2) -- (4,2) -- (4,6) -- (-2,6) -- (-2,4) -- (4,4) ;
        \draw [dashed] (0,0) -- (4,4) -- (2,6) -- (0,4) ++ (-2,2) -- (2,2) ++ (-4,2) -- (0,6) -- (4,2) ++ (0,4) -- (0,2) -- (2,0);
        \draw (4,7) node { ${\displaystyle \Omega}$};
        \draw (4.5,4) node { ${\displaystyle \Oep}$};
\end{tikzpicture}
\caption{An illustration of \ref{hyre} in the cartesian case for $d=2$.}
\label{illhy2}
\end{figure}

  
We must also impose some technical assumptions on $N^{\ep}$ and $E^{\ep}$. 

\begin{hyp} \label{hy5} There exists $N \in \NN, D=\{v_k\}_{k=1,\dots,N} \in C^{0,\alpha}(\RR^d,\Ss)^N$ and $\{c_k\}_{k=1,\dots,N} \in C^1(\Ob, \RR_+^*)^N$ with $\alpha > d/p$ such that $E^{\ep}$ weakly converges in the sense that
\[
\lim_{\ep \rightarrow 0^+} \sum_{(x,e) \in E^{\ep}} |e|^d \varphi \left (x, \frac{e}{|e|} \right ) = \int_{\OS} \varphi(x,v) \: \theta(dx,dv), \forall \varphi \in C(\ObS),
\]
where $\theta \in \mathcal{M}_+(\Omega \times \Ss)$ is of the form 
\[ 
\theta(dx,dv)= \sum_{k=1}^N c_k(x) \delta_{v_k(x)} dx.
\]
\end{hyp}

The $v_k$'s are the possible directions in the continuous model. We have to keep in mind that for every $x \in \RR^d$, the $v_k(x)$ are not necessarily pairwise distinct. The requirement $v_k \in C^{0,\alpha}(\RR^d)$ with $\alpha > d/p$ is technical and will in particular be useful to prove \ref{lem3}. In our examples, the sets of directions are constant and we have
\begin{align*}
D_c & = \{v_{c_1}=(1,0),v_{c_2}=(0,1),v_{c_3}=(-1,0),v_{c_4}=(0,-1)\}, \\
D_t & = D_h =  \{v_{t_k} = e^{i \pi /6} \cdot e^{i \pi (k-1) /3}\}_{k=1,\dots, 6}. 
\end{align*}

The $c_k$'s are the volume coefficients. In our examples, they are constant and do not depend on $k$: 
\[
c_c = 1, c_t = \frac{2}{\sqrt{3}} \text{ and } c_h = \frac{2}{3 \sqrt{3}}.
\]
We notice that the $c_l$'s are different. Indeed, a square whose side length is $\ep$ does not have the same area as a hexagon whose side length is $\ep$. The next assumption imposes another condition on the directions $v_k$'s. 

 \begin{hyp} \label{hy4}
There exists a constant $C >0$ such that for every $(x,z, \xi) \in \RR^d \times \Ss \times \RR_+^N$, there exists $\bar{Z} \in \RR_+^N$ such that $|\bar{Z}| \leq C$ and
\beq \label{min}
\bar{Z} \cdot \xi = \min \{Z \cdot \xi; Z = (z_1, \dots, z_N) \in \RR_+^N \text{ and } \sum_{k=1}^N z_k v_k(x) = z \}.
\eeq
 \end{hyp}
     
This means that the family $\{v_k\}$ is positively generating and that for every $(x,z) \in \RR^d \times \in \RR^d$, a conical decomposition of $z$, not too large compared to $z$, is always possible in the family $D(x)$. In the cartesian case, for all $z \in \RR^2$, we have an unique conical decomposition of $z$ in $D_c$ while in the other examples, we always have the existence but not the uniqueness. The existence of a controlled minimizer allows us to keep some control over "the" decomposition of $z \in \RR^d$ in the family of directions. We now see a counterexample that looks like the cartesian case. We take $N=4, d=2, v_1 = (1,0), v_3 = (-1,0)$ and $v_4 = (0,1)$ (these directions are constant). We assume that there exists $x_0$ in $\Omega$ such that $v_2 (x) \rightarrow - v_4 = (0,1)$ as $x \rightarrow x_0$ and $v_{21}(x) > 0$ for $x \neq x_0$ where $v_2(x)=(v_{21}(x), v_{22}(x))$. Then for every $z=(z_1,z_2) \in \RR^2$ such that $z_1 >0$, we can write $z=\lambda_2(x)  v_2(x) + \lambda_4 (x) v_4(x)$ for $x$ close enough to $x_0$ with 
\[
\lambda_2(x)= \frac{z_1}{v_{21}(x)} \text{ and } \lambda_4(x) = z_1 \frac{v_{22}(x)}{v_{21}(x)} - z_2. 
\]
Then $\lambda_2(x)$ and $\lambda_4(x) \rightarrow +\infty$ as $x \rightarrow x_0$. For $\xi = (1,0,1,0)$, the value of \eqref{min} always is zero but the only decomposition that solves the problem is not controlled.  This example shows that the existence of a controlled decomposition does not imply that the minimal decomposition is controlled. However this assumption is still natural since we want to control the flow on each arc in order to minimize the transport cost.  

\begin{hyp} \label{hy6} 
Up to a subsequence, $\Eep$ may be partitioned as $\Eep = \bigsqcup_{k=1}^N \Eep_k$ such that for every $k=1,\dots,N,$ one has
\begin{equation} \label{H4}
\lim_{\ep \rightarrow 0^+} \sum_{(x,e) \in \Eep_k}  |e|^d \varphi \left (x, \frac{e}{|e|} \right ) = \IO c_k(x) \varphi(x, v_k(x)) \: dx, \: \forall \varphi \in C(\OS, \RR).
\end{equation}
and for every $(x,e) \in \Eep_k$, 
\begin{equation} \label{H5}
\left |\frac{e}{|e|} - v_k(x) \right | = O(1).
 \end{equation} 
 \end{hyp}
 
 This hypothesis is natural. Indeed, for every $x \in \Omega$ the condition $c_k(x) >0$ implies that for $\ep >0$ small enough, there exists an arc $(y,e)$ in $\Eep$ such that $y$ is near $x$ and $e/|e|$ close to $v_k(x)$. We will use it particularly to prove \ref{lem4}. 
The next assumption is more technical and will specifically serve to apply \ref{lem3}. 
 
 \begin{hyp} \label{hy8}
For $\ep > 0$, there exists $d$ (finite) sets of paths $C\Ep_1, \dots, C\Ep_d$ and $d$ linearly independent functions $e_1, \dots, e_d : \Omega \rightarrow \RR^d$ such that for every $x \in \Omega, i=1,\dots,d$, $e_i (x) = \sum_k \alpha_k^i c_k (x) v_k (x)$ where for $k=1, \dots, N$, $\alpha_k^i$ is constant and equal to $0$ or $1$ so that for $i=1,\dots,d$, we have 
\[
\bigcup_{ \sigma \in C\Ep_i} \{ (x,e) \subset \sigma \} = \{ (x,e) \in \Eep / \exists k \in \{ 1, \dots, N \}, \alpha_k^i =1\text{ and } (x,e) \in \Eep_k \}.
\]

For every $(\sigma, \sigma') \in C\Ep_i \times C\Ep_i$, if $\sigma \neq \sigma'$ then $\sigma \cap \sigma' = \emptyset$. We assume 
\[
\max_{\sigma = (y_0, \dots, y_{L(\sigma)}) \in C\Ep_i}(\text{dist}(y_0, \partial \Omega), \text{dist}(y_{L(\sigma)}, \partial \Omega)) \rightarrow 0 \text{ as } \ep \rightarrow 0.
\]
and
\[
| |y_{k+1} -y_k|^{d-1} - |y_k - y_{k-1}|^{d-1} | = O(\ep^d)
\]
for $\sigma = (y_0, \dots, y_{L(\sigma)}) \in C\Ep_i$ and $k= 2, \dots, L(\sigma) - 1$.
\end{hyp}

Formally speaking, it means that we partition points into a set of disjoint paths whose extremities tend to the boundary of $\Omega$ and such that we may do a change of variables for the derivates. For $\vp$ a regular function, $\nabla \vp$ then becomes $(\partial_{e_1} \vp, \dots, \partial_{e_d} \vp)$. Note that the condition on the modules is a sort of volume element in spherical coordinates, $dV$ being approximately $r^{d-1} dr$. In our examples, this requirement is trivial since arcs length is always equal to $\ep$ in $\Oep$. It will allow us to prove some statements on functions. More specifically, we will need to show that some functions have (Sobolev) regularity properties. For this purpose, it will be simpler to use these $e_i$. In the cartesian case ($d=2$), we can simply take $e_{c_1}= v_{c_1}$ and $e_{c_2}=v_{c_2}$ and in the triangular one, we can write $e_{t_1} = c_t v_{t_1}$ and $e_{t_2} = c_t v_{t_2}$. In the hexagonal case, we can for instance define $e_{h_1} = c_h ( v_{h_1} + v_{h_2})$ and $e_{h_2} = c_h (v_{h_2} + v_{h_3})$.  This last example is illustrated by \ref{illhy6} (with $\Omega$ being a circle). We can notice that some arcs are in paths $\sigma \in C\Ep_1$ and $\sigma' \in C\Ep_2$. 

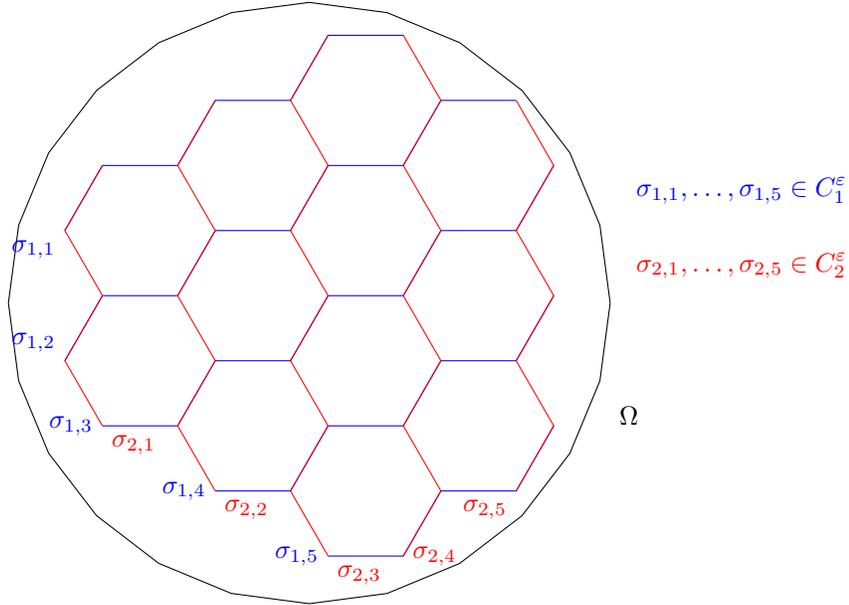
\begin{figure}
\begin{tikzpicture}
        \draw [domain=0:2] plot({3.25+4*cos(deg(\x*pi))},{2.5+4*sin(deg(\x*pi))});  
\draw[blue] (0,0) ++ (60:1cm) node[left] {$\sigma_{1,3}$} --++(0:1cm) --++ (60:1cm) --++(0:1cm) --++ (60:1cm) --++(0:1cm) --++ (60:1cm) --++(0:1cm) --++ (60:1cm);
\draw[blue] (0,0) ++ (60:1cm) ++ (120:1cm) node[above left] {$\sigma_{1,2}$} --++ (60:1cm) --++(0:1cm) --++ (60:1cm) --++(0:1cm) --++ (60:1cm) --++(0:1cm) --++ (60:1cm) --++(0:1cm);
\draw[blue] (0,0) ++ (60:1cm) ++ (120:1cm) ++ (60:1cm) ++ (120:1cm) node[below left] {$\sigma_{1,1}$} --++ (60:1cm) --++(0:1cm) --++ (60:1cm) --++(0:1cm) --++ (60:1cm) --++(0:1cm);
\draw[blue] (0,0) ++ (60:1cm) ++ (0:1cm) ++ (-60:1cm) node[left] {$\sigma_{1,4}$} --++ (0:1cm) --++ (60:1cm) --++(0:1cm) --++ (60:1cm) --++(0:1cm) --++ (60:1cm);
\draw[blue] (0,0) ++ (0:3cm) ++ (-60:1cm) node[left] {$\sigma_{1,5}$} --++ (0:1cm) --++ (60:1cm) --++ (0:1cm) --++ (60:1cm);
\draw[red] (0,0) ++ (60:1cm) node[below right] {$\sigma_{2,1}$} --++ (120:1cm) --++ (60:1cm) --++ (120:1cm) --++ (60:1cm);
\draw[red] (0,0) ++ (0:2cm) node[below right] {$\sigma_{2,2}$} --++ (120:1cm) --++ (60:1cm) --++ (120:1cm) --++ (60:1cm) --++ (120:1cm) --++ (60:1cm);
\draw[red] (0,0) ++ (0:3cm) ++ (-60:1cm) node[below right] {$\sigma_{2,3}$} --++ (120:1cm) --++ (60:1cm) --++ (120:1cm) --++ (60:1cm) --++ (120:1cm) --++ (60:1cm) --++ (120:1cm) --++ (60:1cm);
\draw[red] (0,0) ++ (0:3cm) ++ (-60:1cm) ++ (0:1cm) node[right] {$\sigma_{2,4}$} --++ (60:1cm) --++ (120:1cm) --++ (60:1cm) --++ (120:1cm) --++ (60:1cm) --++ (120:1cm) --++ (60:1cm) --++ (120:1cm);
\draw[red] (0,0) ++ (0:6cm) node[below left] {$\sigma_{2,5}$} --++ (60:1cm) --++ (120:1cm) --++ (60:1cm) --++ (120:1cm) --++ (60:1cm) --++ (120:1cm);
        \draw (7.5,1) node { ${\displaystyle \Omega}$};
        \draw[blue] (9,4) node {${\displaystyle \sigma_{1,1}, \dots, \sigma_{1,5} \in C\Ep_1}$};
        \draw[red] (9,3) node {${\displaystyle \sigma_{2,1}, \dots, \sigma_{2,5} \in C\Ep_2}$};
\end{tikzpicture}
\caption{An illustration of \ref{hy6} in the hexagonal case for $d=2$}
\label{illhy6}
\end{figure}

All these structural hypothesis on $\Omega\Ep$ are satisfied in our three classical examples. The cartesian one is the most obvious and the hexagonal one is a subcase of the triangular one. 
The following assumption is on the transport plan. 
 
 \begin{hyp} \label{hy1}
 $(\ep^{\frac{d}{2} - 1} \gamma^{\ep})_{\ep > 0}$ weakly star converges to a finite nonnegative measure $\gamma$ on $\overline{\Omega} \times \overline{\Omega}$ in the sense that the family of discrete measures \\ 
 $\sum_{(x,y) \in N\Ep \times N\Ep} \ep^{\frac{d}{2} - 1} \gamma^{\ep}(x,y) \delta_{(x,y)}$ weakly star converges to $\gamma$: 
\begin{equation} \label{H1}
 \lim_{\ep \rightarrow 0^+} \sum_{(x,y) \in {N^{\ep}}^2} \ep^{\frac{d}{2} - 1} \gamma^{\ep}(x,y) \varphi(x,y) = \int_{\overline{\Omega} \times \overline{\Omega}} \varphi \: d\gamma; \: \: \forall \varphi \in C(\overline{\Omega} \times \overline{\Omega}).
\end{equation}
\end{hyp} 

The next assumption focuses on the congestion functions $g\Ep$.
 \begin{hyp} \label{hy2}
 $g\Ep$ is of the form
 \begin{equation} \label{H2}
 g^{\ep}(x,e,m)= |e|^{d/2} g \left (x, \frac{e}{|e|}, \frac{m}{|e|^{d/2}} \right), \: \forall \ep >0, (x,e) \in \Eep, m \geq 0
 \end{equation}
 where $g : \OS \times \mathbb{R}_+ \mapsto \mathbb{R}$ is a given continuous, nonnegative function that is increasing in its last variable. 
 \end{hyp}
 
We then have
  \[
 G^{\ep}(x,e,m)= |e|^d G \left (x, \frac{e}{|e|}, \frac{m}{|e|^{d/2}} \right)  \text{ where } G(x, v, m) := \int_0^m g(x, v, \alpha) d \alpha 
 \]
 and
 \[
 H^{\ep}(x,e,t)= |e|^d H \left (x, \frac{e}{|e|}, \frac{t}{|e|^{d/2}} \right) \text{ where } H(x,v,\cdot) = (G(x,v, \cdot))^{\star}
 \]
 i.e. for every $\xi \in \RR_+$ : 
 \[
 H \left (x,\frac{e}{|e|},\xi \right ) = \sup_{m \in \RR_+} \left \{ m \xi - G \left (x,\frac{e}{|e|},m \right ) \right \}.
 \]
For every $(x,v) \in \ObS$, $H(x,v, \cdot)$ is actually strictly convex since $G(x,v, \cdot)$ is $C^1$ (thanks to \ref{hy2}). Note that in the case $d=2$, we have 
 \[
 g\Ep(x,e,m) = |e| g \left ( x,e, \frac{m}{|e|} \right ) \text{ for all } \ep >0, (x,e) \in \Eep, m \geq 0.
 \]
 It is natural and means that the traveling time on an arc of length $|e|$ is of order $|e|$ and depends on the flow per unit of length i.e. $m/|e|$. For the general case, we have extended this assumption. The exponent $d/2$ is not very natural, it does not represent a physical phenomenon but it allows us to obtain the same relation between $G$ and $H$, that means, $H(x,e,\cdot)$ is the Legendre transform of  $G(x,e,\cdot)$. Moreover, we may approach some integrals by sums. We can also note that there is actually no $\ep$-dependence on the $g\Ep$. 
 
 We also add assumptions on $H$: 
 \begin{hyp} \label{hy3} $H$ is continuous with respect to the first two arguments and there exists $p>d$ and two constants $0 < \lambda \leq \Lambda$ such that for every $(x,v,\xi) \in \ObS \times \RR_+$ one has
 \begin{equation} \label{H3}
 \lambda (\xi^p -1) \leq H(x,v, \xi) \leq \Lambda (\xi^p +1).
 \end{equation}
 \end{hyp}
 
 The $p$-growth is natural since we want to work in $L^p$ in the continuous limit. The condition $p >d$ has a technical reason, that will allow us to use  Morrey's inequality. That will be crucial to pass to the limit in the nonlocal term that contains $T\Ep_{\mathbf{t\Ep}}$.
 
\subsection{The limit functional}

In view of the previous paragraph and in particular \ref{hy2}, it is natural to rescale the arc-times $\mathbf{t\Ep}$ by defining new variables 
\beq \label{xep}
\xep(x,e)= \frac{t^{\ep}(x,e)}{|e|^{d/2}} \: \text{ for all } (x,e) \in E^{\ep},
\eeq
i.e. for every $(x,e) \in E^{\ep},$
\[
\xep(x,e) = \frac{g^{\ep}(x,e, m\Ep(x,e))}{|e|^{d/2}} = g \left (x, \frac{e}{|e|}, \frac{m\Ep(x,e)}{|e|^{d/2}} \right ).
\]
Then rewrite the formula $(\ref{D1})$ in terms of $\xep$ as:
 \begin{equation} \label{pd1}
 \inf_{\xep \in \mathbb{R}_+^{\#E^{\ep}}} J^{\ep}(\xep) := I_0^{\ep}(\xep) - I_1^{\ep}(\xep)
 \end{equation}

where
\begin{equation} \label{pd1.1}
I_0^{\ep}(\xep) := \sum_{(x,e) \in E^{\ep}} |e|^d H \left (x, \frac{e}{|e|}, \xep (x,e) \right )
\end{equation}
and
\begin{equation} \label{pd1.2}
I_1^{\ep}(\xep) := \sum_{(x,y) \in {N^{\ep}}^2} \gamma^{\ep}(x,y) \left ( \min_{\sigma \in C_{x,y}^{\ep}} \sum_{(z,e) \subset \sigma} |e|^{d/2} \xep(z,e) \right ) .
\end{equation} 

In view of \ref{hy3}, let us denote
\[
L_+^p (\theta) := \{ \xi \in L^p(\OS, \theta), \xi \geq 0\}.
\]
It is natural to introduce 
\beq \label{I0}
I_0(\xi) := \IOS H(x, v, \xi(x, v)) \theta(dx, dv), \: \forall \xi \in L_+^p (\theta),
\eeq
as the continuous limit of $I_0\Ep$. It is more involved to find the term that plays the same role as $I_1\Ep$ since we must define some integrals on paths. First rearrange the second term \eqref{pd1.2}, that is, as an integral. Let $\xep \in \RR_+^{\#\Eep}, (x,y) \in \Nep \times \Nep$ and $\sigma \in C_{x,y}^{\ep}$, $\sigma = (x_0, ..., x_{L(\sigma)})$. 
Let us extend $\sigma$ on $[0,L(\sigma)]$ in a piecewise affine function by defining  $\sigma(t)=\sigma(k) +(t-k) (\sigma(k+1) - \sigma(k))$ and $\xep$ in a piecewise constant function : $\xep(\sigma(t), \dot{\sigma}(t)) = \xep(\sigma(k), \sigma(k+1) - \sigma(k))$ for $t \in [k,k+1]$. For every $(x,e) \in \Eep$ we call $\Psi\Ep(x, e)$ the "canonical" decomposition of $ e$ on $D(x)$. More precisely, recalling \ref{hy6}, for $(x,e) \in \Eep$, there exists $k_{(x,e)} \in \{1, \dots, N\}$ such that $(x,e) \in \Eep_{k_{(x,e)}}$ and then we set 
\[
\Psi\Ep(x, e) = (0,\dots, \underset{k_{(x,e)}\text{th coordinate}}{|e|}, \dots, 0) \in \RR^N.
\]
For $\sigma \in C\Ep$ and $t \in [k,k+1[$, we write $\Psi\Ep(\sigma(t), \dot{\sigma}(t)) = \Psi\Ep(\sigma(k), \sigma(k+1) - \sigma(k))$. Let us also define a function $\boldsymbol{\xep}$ as follows: 
\[
\boldsymbol{\xep}(x) = \sum_{e / (x,e) \in \Eep} \frac{\xep(x,e)}{|e|} \Psi\Ep(x,e) \text{ for } x \in N\Ep.
\]
Let us extend $\boldsymbol{\xep}$ in a piecewise constant function on the arcs of $\Eep$: let $y \in \Omega$ such that there exists $(x,e) \in \Eep$ with $y \in (x,e)$, then we define 
\[
\boldsymbol{\xep}(y)= \sum_{(x,e) \in \Eep / y \in (x,e)} \frac{\xep(x,e)}{|e|} \Psi\Ep(x,e).
\]
This definition is consistent since the arcs that appear in the sum are in some different $\Eep_k$. By abuse of notations, we continue to write $\sigma, \xep$ and $\boldsymbol{\xep}$ for these new functions. Thus we have
\begin{align*}
\sum_{(x,e) \subset \sigma} |e| \xep(x,e) & = \sum_{k=0}^{L(\sigma)-1} | \sigma(k+1) - \sigma(k)| \xep(\sigma(k), \sigma(k+1) - \sigma(k)) \\
& = \sum_{k=0}^{L(\sigma)-1} \Psi\Ep(\sigma(k), \sigma(k+1) - \sigma(k)) \cdot \boldsymbol{\xep}(\sigma(k)) \\
& = \int_0^{L(\sigma)} \Psi\Ep(\sigma(t), \dot{\sigma}(t)) \cdot \boldsymbol{\xep}(\sigma(t)) dt.
\end{align*} 
We then get
\begin{align*}
 \min_{\sigma \in C_{x,y}^{\ep}} \sum_{(x,e) \subset \sigma} |e| \xep(x,e) &= \inf_{\sigma \in C_{x,y}^{\ep}} \int_0^{L(\sigma)} \Psi\Ep(\sigma(t), \dot{\sigma}(t)) \cdot \boldsymbol{\xep}(\sigma(t)) dt \\
& = \inf_{\sigma \in C_{x,y}^{\ep}} \int_0^1 \Psi\Ep(\tilde{\sigma}(t), \dot{\tilde{\sigma}}(t)) \cdot \boldsymbol{\xep}(\tilde{\sigma}(t)) dt
\end{align*} 
where $\tilde{\sigma} : [0,1] \rightarrow \Ob$ is the reparameterization of $\tilde{\sigma}(t) = \sigma(L(\sigma)t), t \in [0,1].$ For every $x \in \overline{\Omega}, z \in \RR^d$ let us define
\[
 A_x^z = \left \{ Z \in \mathbb{R}_+^N, Z=(z_1,...,z_N) / \sum_{k=1}^N z_k v_k(x) = z \right \}.
 \]
 Then for every $\xi \in C(\ObS, \mathbb{R}_+)$, define
 \begin{align*}
 c_{\xi}(x,y) & = \inf_{\sigma \in C_{x,y}} \left \{ \int_0^1 \inf_{X \in A_{\sigma(t)}^{\dot{\sigma}(t)}} \left ( \sum_{k=1}^N x_k \xi(\sigma(t), v_k(\sigma(t))) \right ) \: dt \right \} \\
 & = \inf_{\sigma \in C_{x,y}} \inf_{\rho \in \mathcal{P}_{\sigma}} \int_0^1 \left ( \sum_{k=1}^N  \xi(\sigma(t), v_k(\sigma(t))) \:  \rho_k(t) \right ) \: dt \\ 
 & =  \inf_{\sigma \in C_{x,y}} \inf_{\rho \in \mathcal{P}_{\sigma}} L_{\xi}(\sigma, \rho),
  \end{align*}
where
  \[ 
  \mathcal{P}_{\sigma} = \left \{ \rho : t \in [0,1] \rightarrow \rho(t) \in \RR_+^N / \dot{\sigma}(t) = \sum_{k=1}^N v_k(\sigma(t)) \: \rho_k (t) \text{ a.e. } t \right \},
  \]  
  \[
  L_{\xi}(\sigma, \rho) = \int_0^1 \left ( \sum_{k=1}^N  \xi(\sigma(t), v_k(\sigma(t))) \:  \rho_k(t) \right ) \: dt 
  \]
  and $C_{x, y}$ is the set of absolutely continuous curves $\sigma$ with values in $\Ob$ and such that $\sigma(0)=x$ and $\sigma(1) = y$. For every $x \in \Omega$ and $z \in \RR^d$, there is no a priori uniqueness of the decomposition of $z$ in the family $\{v_k(x) \}_k$ so that we have to take the infimum over all possible decompositions. The definition of $\Pc_{\sigma}$ takes in account this constraint. For every $\rho \in \Pc_{\sigma}$ and $t \in [0,1]$ the terms $\rho_k(t)$ are the weights of $\dot{\sigma}(t)$ in the family $\{v_k(\sigma(t) \}$. $(\sigma, \rho)$ is a sort of generalized curve. It will allow us to distinguish between different limiting behaviors. A simple example is the following approximations, $\ep > 0$ : 
   
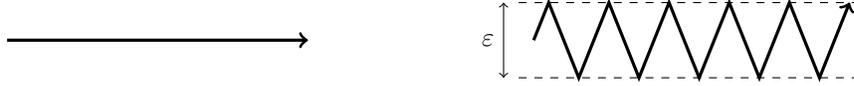
\begin{tikzpicture}
\draw [very thick, ->] (0,0) -- (4,0);
\draw [very thick, ->] (7,0) --++ (0.2, 0.5) --++ (0.4,-1) --++ (0.4,1) --++ (0.4,-1) --++ (0.4,1) --++ (0.4,-1) --++ (0.4,1) --++ (0.4,-1) --++ (0.4,1) --++ (0.4,-1) --++ (0.4,1) ;
\draw [dashed] (6.8,0.5) --++ (4.6,0);
\draw [dashed] (6.8,-0.5) --++ (4.6,0);
\draw [<->] (6.6,0.5) --++ (0,-1);
\draw (6.6,0) node [left] {${\displaystyle \ep}$};
\end{tikzpicture}

It is the same curve $\sigma$ in both case but the $\rho$ are different. Indeed, in the left case, we have only one direction that is $(1,0)$, it is a line. In the right one, we have two directions that tend to $(0,1)$ and $(0,-1)$ as $\ep$ tends to $0$, these are oscillations.

  Now, our aim is to extend the definition of $c_{\xi}$ to the case where $\xi$ is only $L_+^p(\theta)$. We will strongly generalize the method used in \cite{carlier2008optimal} to the case of generalized curves. First, let us notice that we may write $c_{\xi}$ in another form: 
  \[
c_{\xi}(x,y) = \inf_{\sigma \in C_{x,y}} \tilde{L}_{\xi}(\sigma) \: \text{ for } \xi \in C(\ObS, \mathbb{R}_+),
\]
where 
\beq \label{Ltxs}
\Ltxs = \int_0^1 \Phi_{\xi} (\sigma(t), \dot{\sigma}(t)) \: dt
\eeq  
with for all $x \in \Ob$ and $y \in \RR^d$, $\Phi_{\xi}(x,y)$ being defined as follows: 
\begin{align*}
\Phi_{\xi}(x,y) & = \inf_{Y \in \RR_+^N} \left \{ \sum_{k=1}^N y_k \xi(x,v_k(x)) : Y = (y_1, \cdots, y_N) \in A_x^y \right \} \\
& = \inf_{Y \in A_x^y} Y \cdot \xig(x), 
\end{align*}
where $\xig(x) = (\xi(x, v_1(x)), \dots,\xi(x, v_N(x)))$.  
 
  The next lemma shows that $\Phi_{\xi}$ defines a sort of Finsler metric. It is an anisotropic model but $\Phi_{\xi}$ is not even and so $c_{\xi}$ is not symmetric. Moreover $c_{\xi}$ is not necessarily positive between two different points so that $c_{\xi}$ is not a distance. However $\Phi_{\xi}(x,\cdot)$ looks like a norm that depends on the point $x \in \Ob$. Its unit ball is a polyhedron in $\RR^d$ that changes with $x$. Formally speaking, the minimizing element $Y=(y_1, \dots, y_N)$ represents the coefficients for the "Finsler distance" $c_{\xi}$. 
  
\begin{lemme} \label{lempre} 
Let $\xi \in C(\ObS, \mathbb{R}_+)$. Then one has: 
\begin{enumerate}
\item The function $\Phi_{\xi}$ is a minimum and continuous. 
\item For all $x \in \Ob$, $\Phi_{\xi} (x, \cdot)$ is homogeneous of degree $1$ and convex. 
\item If $\{ \xi_n \}$ is a sequence in $C(\ObS, \mathbb{R}_+)$ and $x \in \Ob$ such that $\{ \boldsymbol{\xi_n}(x) \}$ converges to $\boldsymbol{\xi}(x)$ then $\{ \Phi_{\xi_n} (x,y) \}$ converges to $\Phi_{\xi}(x,y)$ for all $y \in \RR^d$.
\end{enumerate}
\end{lemme} 

The proof is left to the reader.

Due to \ref{lempre}, we may reformulate \eqref{Ltxs} :
\[
\Ltxs = \int_0^1 \Phi_{\xi} (\sigma(t), \dot{\sigma}(t)) \: dt = \int_0^1 |\dot{\sigma}(t)| \Phi_{\xi} \left (\sigma(t), \frac{\dot{\sigma}(t)}{|\dot{\sigma}(t)|} \right) dt.
\]
 The following lemma gives a H\"older estimate for $c_{\xi}$ and extends Proposition $3.2$ in \cite{carlier2008optimal}.
\begin{lemme} \label{prop3.2}
There exists a nonnegative constant $C$ such that for every $\xi \in C(\ObS, \RR_+)$ and every $(x_1, x_2, y_1, y_2) \in \Omega^4$, one has
\begin{equation} \label{eq3.2} 
| c_{\xi} (x_1,y_1) - c_{\xi}(x_2,y_2) | \leq C \| \xi \|_{L^p(\theta)} (|x_1 - x_2|^{\beta} + |y_1 - y_2|^{\beta}),
\end{equation}
where $\beta = 1 - d/p$.
So if $(\xi_n)_n \in C(\ObS, \RR_+)^{\NN}$ is bounded in $L^p(\theta)$, then $(c_{\xi_n})_n$ admits a subsequence that converges in $C(\ObOb, \RR_+)$.
\end{lemme}

 For every $\xi \in L_+^p (\theta)$, let us define
\beq \label{cbx}
\cbx(x,y) = \sup \{c(x,y) : c \in \mathcal{A} (\xi) \}, \; \forall (x,y) \in \overline{\Omega} \times \overline{\Omega} 
\eeq

where
\[
 \mathcal{A}(\xi) = \{ \lim_n c_{\xi_n} \text{ in } C(\overline{\Omega} \times \overline{\Omega}) : (\xi_n)_n \in C(\overline{\Omega} \times \Ss, \mathbb{R}_+)^{\NN}, \xi_n \rightarrow \xi \text{ in } L^p (\theta) \}.
 \]

We will justify in the next section that for every $\xi \in L_+^p (\theta)$, the continuous limit functional is
\begin{equation} \label{pc1}
J(\xi) := I_0(\xi) - I_1(\xi) = \IOS H(x, v, \xi(x, v)) \theta(dx, dv) - \int_{\overline{\Omega} \times \overline{\Omega}} \cbx d\gamma.
\end{equation}

The two following lemmas are generalizations of Lemmas $3.4$ and $3.5$ in \cite{carlier2008optimal}:

\begin{lemme} \label{convcxin}
If $\xi \in L_+^p (\theta)$ then there exists a sequence $(\xi_n)_n$ in $C(\ObS, \RR_+)$ such that $c_{\xi_n}$ converges to $\cbx$ in $C(\ObOb)$ as $n \rightarrow \infty$.  
\end{lemme}
%

\begin{lemme} \label{lem1}
If $\xi \in C(\ObS, \RR_+)$ then $c_{\xi} = \cbx$. 
\end{lemme}

 \subsection{The $\Gamma$-convergence result}
 
 We will prove that the problem \eqref{pc1} is the continuous limit of the discrete problems \eqref{pd1} in the $\Gamma$-convergence sense. The $\Gamma$-convergence theory is a powerful tool to study the convergence of variational problems (convergence of values but also of minimizers) depending on a parameter. Here we want to study problems depending on a scale parameter (which is $\ep$), it is particularly well suited. References for the general theory of $\Gamma$-theory and many applications are the books of Dal Maso \cite{dal1993introduction} and Braides \cite{braides2002gamma}.
 
 First let us define weak $L^p$ convergence of a discrete family $\xep \in \RR_+^{\#\Eep}$.

\begin{defi} \label{def1}
For $\ep > 0$, let $\xep \in  \mathbb{R}_+^{\#\Eep}$ and $\xi \in L_+^p (\theta)$, then $\xep$ is said to weakly converge to $\xi$ in $L^p$ ($\xep \rightarrow \xi$) if : \\
\begin{enumerate}
\item There exists a constant $M>0$ such that $\forall \ep >0$, one has 
\[ 
\| \xep \|_{\ep,p} := \left( \sum_{(x,e) \in \Eep} |e|^d \xep (x, e) ^p \right) ^{1/p} \leq M,
\]
\item For every $\varphi \in C(\ObS,\mathbb{R})$, one has
\[
\lim_{\ep \rightarrow 0^+}  \sum_{(x,e) \in \Eep} |e|^d \varphi \left (x, \frac{e}{|e|} \right ) \xep (x, e) = \IOS \varphi(x, v) \xi (x, v) \theta (dx, dv). 
\] 
\end{enumerate}
\end{defi}

\begin{defi}
For $\ep > 0$, let $F^{\ep} :  \mathbb{R}_+^{\#\Eep} \rightarrow \mathbb{R} \cup \{+ \infty \}$ and $ F : L_+^p(\theta) \rightarrow \mathbb{R} \cup \{+ \infty \}$, then the family of functionals $(F^{\ep})_{\ep}$ is said to $\Gamma$-converge (for the weak $L^p$ topology) to $F$ if and only if the following two conditions are satisfied: \\
\begin{enumerate}
\item (\text{$\Gamma$-liminf inequality}) $\forall \xi \in L_+^p (\theta), \xep \in  \mathbb{R}_+^{\#\Eep}$ such that $\xep \rightarrow \xi$, one has 
\[
\liminf_{\ep \rightarrow 0^+} F^{\ep}(\xep) \geq F(\xi), 
\]
\item (\text{$\Gamma$-limsup inequality}) $\forall \xi \in L_+^p (\theta)$, there exists $\xep \in  \mathbb{R}_+^{\#\Eep}$ such that $\xep \rightarrow \xi$ and
\[
\limsup_{\ep \rightarrow 0^+} F^{\ep}(\xep) \leq F(\xi), 
\]
\end{enumerate}
\end{defi}

Now, we can state our main result, whose complete proof will be performed in the next section: 

\begin{theo} \label{th1} Under all previous assumptions, the family of functionals
$(J^{\ep})_{\ep} \; \Gamma$-converges (for the weak $L^p$ topology) to the functional $J$ defined by \eqref{pc1}.
\end{theo}

Classical arguments from general $\Gamma$-convergence theory allow us to have the following convergence result : 

\begin{coro} \label{coro1}
Under all previous assumptions, the problems \eqref{pd1} for all $\ep >0$ and \eqref{pc1} admit solutions and one has: 
\[ \min_{\xep \in  \mathbb{R}_+^{\#\Eep}} J^{\ep}(\xep) \rightarrow \min_{\xi \in L_+^p (\theta)} J(\xi). \]
Moreover, if for any $\ep > 0$, $\xep$ is the solution of the minimization problem \eqref{pd1} then $\xep \rightarrow \xi$ where $\xi$ is the minimizer of $J$ over $L_+^p (\theta)$.
\end{coro}

\bpr
First, due to \eqref{H3} we have
\begin{align*}
I_0\Ep (\xep) & \geq \lambda \sum_{(x,e) \in \Eep} |e|^d \left ( \xep (x, e) ^p -1 \right ) \\
& = \lambda \| \xep \|_{\ep, p}^p - \lambda \sum |e|^d \\
& \geq \lambda \| \xep \|_{\ep, p}^p - C,
\end{align*}
since from \ref{hy6} it follows that
\[
\sum_{(x,e) \in \Eep} |e|^d  \rightarrow \IOS \theta (dx, dv) = \sum_{k=1}^N \IO c_i(x) \: dx = C. 
\]
To estimate the other term $I_1\Ep (\xep)$, let us write
\[
c\Ep (x,y) = \min_{\sigma \in C_{x,y}\Ep} \sum_{(z,e) \subset \sigma} |e| \xep  ( z, e).
\]
$\forall x_0 \in N\Ep, \forall x,y  \in N\Ep$ two neighboring nodes, we have 
\begin{align*}
|c\Ep (x_0,x) - c\Ep (x_0,y) | & \leq  \max_{e / (x,e) \in \Eep} |e| \xep ( x, e) \\
& \leq \ep \max_{e / (x,e) \in \Eep} \xep ( x, e).
\end{align*}
Then thanks to \ref{dismor}, we have for every $x,y \in N\Ep$, 
\[
|c\Ep (x_0,x) - c\Ep (x_0,y) | \leq  C \left \| \max_{e / (\cdot,e) \in \Eep} \xep ( \cdot, e) \right \|_{\ep,p},
\]
hence with $x_0=x$ 
\begin{align*}
c\Ep (x,y)  & \leq  C \left \| \max_{e / (\cdot,e) \in \Eep} \xep ( \cdot, e) \right \|_{\ep,p} \\
& \leq C \| \xep \|_{\ep,p}
\end{align*}
so that recalling \eqref{pd1.2}, we have 
\[
|I_1\Ep(\xep) | \leq C \| \xep \|_{\ep,p} \sum_{x,y \in N\Ep} \ep^{d/2-1} \gamma\Ep (x,y) \leq C \| \xep \|_{\ep,p}, 
\]
because it follows from \ref{hy1} that
\[
\sum_{x,y \in N\Ep} \ep^{d/2-1} \gamma\Ep (x,y) \rightarrow \IOO d \gamma \text{ as } \ep \rightarrow 0^+.
\]
In particular, we get the equi-coercivity estimate
\[
J\Ep(\xep) \geq C(\| \xep \|_{\ep, p}^p - \| \xep \|_{\ep, p} -1).
\]
Since $J\Ep$ is continuous on $ \RR_+^{\#\Eep}$, this proves that the infimum of $J\Ep$ over $ \RR_+^{\#\Eep}$ is attained at some $\xep$ and also that $\| \xep \|_{\ep, p}$ is bounded, in particular, we can define for $\ep >0$ the following Radon measure $M_{\ep}$  :
\[
\langle M_{\ep}, \varphi \rangle := \sum_{(x,e) \in \Eep} |e|^d \varphi \left ( x, \frac{e}{|e|} \right ) \xep ( x, e ), \: \forall \varphi \in C(\ObS, \RR).
\]
Due to H\"older inequality we have for every $\ep > 0$ and $\vp \in C(\ObS, \RR)$, 
\begin{equation} \label{eqM}
|\langle M_{\ep}, \varphi \rangle| \leq C \| \vp \|_{\ep, q}
\end{equation}
where $q = p/(p-1)$ is the conjugate exponent of $p$ and the semi-norm $\| . \|_{\ep, q}$ is defined by:
\[
 \| \vp \|_{\ep,q} := \left( \sum_{(x,e) \in \Eep} |e|^d \vp \left (x, \frac{e}{|e|} \right) ^q \right) ^{1/q}. 
 \]
Because of \ref{hy5} there is a nonnegative constant $C$ such that $ \| \vp \|_{\ep,q} \leq C \| \vp \|_{\infty}$ for every $\vp \in C(\ObS, \RR)$. We deduce from \eqref{eqM} and Banach-Alaoglu's theorem that there exists a subsequence (still denoted by $M_{\ep}$) and a Radon measure $M$ over $\ObS$ with values in $\RR$ to which $M_{\ep}$ weakly star converges. Moreover, due to \eqref{eqM} and \ref{hy5}, we have for every $\vp \in C(\ObS, \RR)$
\[
| \langle M, \vp \rangle | \leq C \lim_{\ep \rightarrow 0^+} \| \vp \|_{\ep, q} = C \| \vp \|_{L^q(\theta)}
\]
which proves that $M$ in fact admits an $L^p(\theta)$-representative denoted by $\xi$. $\xi \in L_+^p (\theta)$ (componentwise nonnegativity is stable under weak convergence) and $\xep \rightarrow \xi$ in the sense of \ref{def1}. We still have to prove that $\xi$ minimizes $J$ over $L_+^p (\theta)$. First from the $\Gamma$-liminf inequality we find that 
\[
J(\xi) \leq \liminf_{\ep \rightarrow 0^+} J\Ep(\xep) =  \liminf_{\ep \rightarrow 0^+} \min_{\xep \in  \mathbb{R}_+^{\#\Eep}} J^{\ep}(\xep).
\]
Let $\zeta \in L_+^p(\theta)$, we know from the $\Gamma$-limsup inequality that there exists a sequence $(\zeta\Ep)_{\ep} \in \mathbb{R}_+^{\#\Eep}$ such that $\zeta\Ep \rightarrow \zeta$ in the sense of \ref{def1} and that
\[
\limsup_{\ep \rightarrow 0^+} J\Ep(\zeta\Ep) \leq J(\zeta). 
\]
Since $\xep$ minimizes $J\Ep$ we have that
\[
J(\xi) \leq \liminf_{\ep \rightarrow 0^+} J\Ep(\xep) \leq \limsup_{\ep \rightarrow 0^+} J\Ep(\xep) \leq \limsup_{\ep \rightarrow 0^+} J\Ep(\zeta\Ep) \leq J(\zeta). 
\]
We can then deduce that $\xi$ minimizes $J$ over $L_+^p (\theta)$ and we have also proved the existence of a minimizer to the limit problem. We also have that
\[
\min_{L_+^p (\theta)} J \leq \liminf_{\ep \rightarrow 0^+} \min_{\xep \in  \mathbb{R}_+^{\#\Eep}} J^{\ep}(\xep) \leq \limsup_{\ep \rightarrow 0^+} \min_{\xep \in  \mathbb{R}_+^{\#\Eep}} J^{\ep}(\xep) \leq J(\zeta), \: \forall \zeta \in L_+^p (\theta)
\]
which provides the convergence of the values of the discrete minimization problems to the value of the continuous one. Furthermore we have convergence of the whole family $\xep$ and not only of a subsequence by the uniqueness of the minimizer $\xi$ of $J$ over $L_+^p (\theta)$ since $J$ (and in particular $I_0$, due to \ref{hy2}) is strictly convex.
 \epr
 
 \section{Proof of the theorem}

\subsection{The $\Gamma$-liminf inequality} \label{subs4.1}

For $\ep > 0$, let $\xep \in  \mathbb{R}_+^{\#\Eep}$ and $\xi \in L_+^p (\theta)$ such that $\xep \rightarrow \xi$ (in the sense of \ref{def1}). In this subsection, we want to prove that
\beq \label{liminf}
\liminf_{\ep \rightarrow 0^+} J\Ep (\xep) \geq J(\xi).
\eeq

We need some lemmas to establish this inequality. The first one concerns the terms $I_0\Ep$ and $I_0$. 

\begin{lemme} \label{lem4.1}
One has 
\[
 \liminf_{\ep \rightarrow 0^+} I_0^{\ep} (\xep) \geq I_0(\xi).
 \]
\end{lemme}

The proof is similar to that of Lemma $4.2$ in \cite{baillon2012discrete}.
%
Now we need the following discrete version of Morrey's inequality to have information about the nonlocal term.

\begin{lemme} \label{dismor}
Let $\theta\Ep \in \RR_+^{\#N\Ep}$ and $\varphi\Ep \in \RR_+^{\#N\Ep}$ such that
\beq \label{eqlem2}
|\varphi\Ep(x)-\varphi\Ep(y)| \leq \ep \theta\Ep(x), \text{ for every } x \in N\Ep \text{ and every } y \text{ neighbor of } x,
\eeq
then there exists a constant $C$ such that for every $(x,y) \in \Eep \times \Eep$, one has
\[
|\varphi\Ep(x)-\varphi\Ep(y) | \leq C \| \theta\Ep \|_{\ep, p} (|x_1 - y_1| + |x_2 - y_2|)^{\beta}
\]
where $\beta = 1 - d/p$ and 
\[
\| \theta\Ep \|_{\ep, p} = \left ( \ep^d \sum_{x \in N\Ep} \theta\Ep (x)^p \right ) ^{1/p}.
\]
\end{lemme}

\bpr
The idea is to linearly interpolate $\vp\Ep$ in order to have a function in $W^{1,p}(\Omega)$ and then to apply Morrey's inequality. 
By recalling \ref{hyre}, let $V\Ep= \text{Conv}(x_1\Ep, ..., x_L\Ep)$ a polytope in the discrete network  $\Oep$ where for $k=1, \dots, L$, $x_k\Ep$ is an neighbor of $x\Ep_{k+1}$ in $N\Ep$, the indices being taken modulo $L$. Let us denote $X\Ep$ the isobarycenter of all these nodes (it is in the interior of $V\Ep$ by assumption) and let us define
\[
\varphi\Ep(X\Ep)= \sum_{k=1}^L \frac{\varphi\Ep(x_k\Ep)}{L}.
\]
Let $\{F_j\Ep\}$ denote the subpolytopes given by \ref{hyre} for $V\Ep$ with $F_j\Ep = \text{Conv}(X\Ep, X\Ep_{j_1}, \dots, X\Ep_{j_d})$ and these $(d+1)$ points that are linearly independent.  Then for every $x$ in $V\Ep$, there exists $j$ such that $x \in F_j\Ep$. $x$ is a conical combination of $X\Ep, X\Ep_{j_1}, \dots, X\Ep_{j_d}$. There exists some unique nonnegative coefficients $\lambda, \lambda_1, \dots, \lambda_d \geq 0$ such that $x= \lambda X\Ep + \sum_{i=1}^d \lambda_i X\Ep_{j_d}$. We set 
\beq \label{intlin}
\vp\Ep(x)= \lambda \vp(X\Ep) + \sum_{i=1}^d \lambda_i \vp \left (X\Ep_{j_d} \right ).
\eeq
We still denote this interpolation by $\vp\Ep$. We then have $\vp\Ep \in W^{1,p}(\Omega)$ with $\| \nabla \varphi\Ep \|_p \leq C \| \theta\Ep \|_{\ep, p}$ (computing $\nabla \varphi\Ep$ is technical and we detail it below). We conclude thanks to Morrey's inequality.


\textbf{Computing} $\nabla \varphi\Ep$: 

We take the notations used in the above proof but we remove the $\ep$-dependence for the sake of simplicity. Taking into account the construction of $\vp\Ep$, we will compute $\nabla \varphi\Ep$ on a subpolytope $F= \text{Conv}(X, X_1, ..., X_d)$ of $V$ where $X$ is the isobarycenter of all nodes $x_k$ in $V\Ep$ and $X, X_1, \dots, X_d$ are linearly independent. We rearrange the $x_k$'s such that $x_k$ is a neighbor of $x_{k+1}$. Let $\mathcal{H}$ be the affine hyperplane  
\[
\mathcal{H} = \left < \left( \begin{array}{c} X \\ \vp\Ep(X) \end{array} \right),  \left( \begin{array}{c} X_1 \\ \vp\Ep(X_1) \end{array} \right), \dots, \left( \begin{array}{c} X_d \\ \vp\Ep(X_d) \end{array} \right) \right >
\] 
i.e. the affine subspace containing $(X, \vp\Ep(X))$ and directed by the vector space generated by the vectors $(X_1-X, \vp\Ep(X_1) - \vp\Ep(X)), \dots, (X_d-X, \vp\Ep(X_d) - \vp\Ep(X))$. It is a hyperplane since the points $X, X_1, \dots, X_d$ are linearly independent. Then there exists some constants $a_0,\dots , a_{d+1}$ such that $\mathcal{H} = \{ (z_1, \dots, z_{d+1}) \in \RR^{d+1}; a_1 z_1 + \dots + a_{d+1} z_{d+1} + a_0 = 0 \}$ with $a_{d+1} \neq 0$ (otherwise $X, X_1, \dots, X_d$ would not be independent). The normal vector to the hyperplane is $\vec{n} = (a_1, \dots, a_{d+1})$ and 
\[
\nabla \vp\Ep = \left( \begin{array}{c} - a_1/a_{d+1} \\  \vdots \\ - a_d/a_{d+1} \end{array} \right).
\]
Without loss of generality we assume $X=0, X_1 = (y_1, 0), X_2 = (y_1^2, y_2, 0), \dots, X_d= (y_1^d, \dots, y_{d-1}^d, y_d)$, with $y_1, \dots, y_d \neq 0$. Then we have for $k=1, \dots, d$
\[
\left( \begin{array}{c} X_k - X \\ \vp\Ep(X_k) - \vp\Ep(X) \end{array} \right) \cdot \vec{n} = 0 = \sum_{i < k} a_i y_i^k + a_k x_k + a_{d+1} (\vp\Ep(X_k) - \vp\Ep(X)),
\]
i.e. 
\beq \label{ak}
- \frac{a_k}{a_{d+1}} = \sum_{i < k} \frac{a_i}{a_{d+1}} \frac{y_i^k }{y_d} + \frac{\vp\Ep(X_k) - \vp\Ep(X)}{y_k}.
\eeq
We must now find an estimate on each of these terms. First for $k=1, \dots, d$, there exists $i_k \in \{1, \dots, L \}$ such that $X_k = x_{i_k}$ and it follows from \eqref{eqlem2} that we have 
\begin{align*}
|\vp\Ep(X_k) - \vp\Ep(X)| & = \left | \sum_{l=1}^L \frac{L-l+1}{L} (\vp\Ep(x_{l+i_k-1})) - \vp\Ep(x_{l+i_k})) \right | \\
& \leq \ep \sum_{l=1}^L \theta\Ep (x_l)
\end{align*}
with the indices taken modulo $L$. Moreover, due to \eqref{estiso}, for $k=1, \dots, d$ and $i<k$, we have $|y_i^k| \leq |X-X_k| \leq C \ep.$ Taking into account the fact that the $X_k$'s are linearly independent, we get $|y_k| \geq C \ep$. Then from \eqref{ak} it follows for $k=1, \dots, d$
\beq \label{a/c}
\left | \frac{a_k}{a_{d+1}} \right | \leq C \sum_{i<k} \left | \frac{a_i}{a_{d+1}} \right | + C \sum_{l=1}^L \theta\Ep (x_l).
\eeq
We finally obtain by an induction on $k$ that
\[
\left | \frac{a_k}{a_{d+1}} \right | \leq C \sum_{l=1}^L \theta\Ep (x_l)
\]
so that due to \eqref{intlin}, $|\nabla \varphi\Ep(x) | \leq C \sum_{l=1}^L \theta\Ep (x_l)$ for every $x$ in the subpolytope $F$. 
We finally conclude with \ref{hyre} that  
 \[
 \| \nabla \varphi\Ep \|_p \leq C \| \theta\Ep \|_{\ep, p},
 \]
 which completes the proof. 
\epr

The discretization of the Morrey inequality is crucial since we may now extend $c\Ep$ on $\Omega \times \Omega$. For every $(x, y) \in N\Ep \times N\Ep$, we define 
\beq \label{cEp}
c\Ep (x,y) = \min_{\sigma \in C_{x,y}^{\ep}} \sum_{(z,e) \subset \sigma} |e| \xep (z,e) = \inf_{\sigma \in C_{x,y}^{\ep}} \int_0^1 \Psi\Ep(\tilde{\sigma}(t), \dot{\tilde{\sigma}}(t)) \cdot \boldsymbol{\xep}(\tilde{\sigma}(t)) dt .
\eeq
By definition, if $x_0 \in \Omega_{\ep}$ and $x$ and $y$ neighbors in $\Oep$, we have
\[
c\Ep(x_0, x) \leq c\Ep (x_0, y) + \ep \max_{ e / (y,e) \in \Eep} \xep(y, e). 
\]
Since $\| \xep \|_{\ep, p}$ is bounded, we deduce from \ref{dismor} that there exists a constant $C$ such that for every $\ep > 0$ we have
\[
| c\Ep(x,y) - c\Ep(x_0, y_0) | \leq C (|x - x_0|^{\beta} + |y - y_0|^{\beta} ), \: \forall \: (x, y, x_0, y_0) \in {N\Ep}^4.
\] 
We can then extend $c\Ep$ to the whole $\ObOb$ (we still denote by $c\Ep$ this extension) by
\[
c\Ep(x, y) := \sup_{(x_0, y_0) \in \Oep \times \Oep} \{ c\Ep (x_0, y_0) - C (|x- x_0|^{\beta} + |y - y_0|^{\beta} ) \}, \: \forall \:  (x, y) \in \ObOb. 
\]
By construction, $c\Ep$ still satisfy the uniform H\"older estimate on the whole $\ObOb$ and since $c\Ep$ vanishes on the diagonal of $\ObOb$, it follows from Arzela-Ascoli theorem that the family $(c\Ep)_{\ep} $ is relatively compact in $C(\ObOb)$. Up to a subsequence, we may therefore assume that there is some $c \in C(\ObOb)$ such that 
\beq \label{4.8}
c\Ep \rightarrow c \text{ in } C(\ObOb) \: \text{ and } c(x,x)=0  \text{ for every } x \in \Ob.
\eeq 
From \ref{hy1} it may be concluded that
\[
I_1\Ep (\xep) \leq \sum_{(x,y) \in N\Ep \times N\Ep} \ep^{\frac{d}{2}-1} c\Ep(x,y) \gamma\Ep(x,y) \rightarrow \IObOb c d\gamma.
\]
In consequence, with \ref{lem4.1}, it remains to prove $c \leq \cbx$ on $\ObOb$. We will show that $c$ is a sort of subsolution in a highly weak sense of an Hamilton-Jacobi equation and we will then conclude by some comparison principle. The end of this paragraph provides a proof of this inequality.

\begin{lemme} \label{lem3}
Let $x_0 \in \Omega, \: \xi \in L_+^p(\theta)$ and $\varphi \in W^{1,p}(\Omega)$ such that $\varphi(x_0) = 0$ (which makes sense since $p>d$ so that $\vp$ is continuous). If for a.e. $x \in \Omega$ one has
\begin{equation} \label{eqlem} 
\nabla \varphi (x) \cdot u \leq \Phi_{\xi}(x,u) = \inf_{X \in A_x^u} \left ( \sum_{k=1}^N x_k \xi(x, v_k(x)) \right ) \text{ for all } u \in \mathbb{R}^d 
\end{equation}
then $\varphi \leq \cbx(x_0,\cdot) \text{ on } \Omega$. 
\end{lemme}

\begin{Rem}
The above assumption \eqref{eqlem} is equivalent to : 
\[ \nabla \varphi (x) \cdot v_k(x) \leq \xi_k(x) = \xi(x, v_k(x)), \: \forall x \text{ a.e.}, \forall k=1, \dots ,N. \]
\end{Rem}

\begin{proof}
The result is immediate if $\varphi \in C^1(\overline{\Omega})$ and $\xi$ is continuous on $\overline{\Omega}$. Indeed, in this case, assumption \eqref{eqlem} is true pointwise and if $x \in \Omega$ and $\sigma$ is an absolutely continuous curve with values in $\overline{\Omega}$ connecting $x_0$ and $x$ then by the chain rule we obtain 
\[
\varphi(x)=\int_0^1 \nabla \varphi(\sigma(t)) \cdot \dot{\sigma}(t) dt \leq \int_0^1\Phi_{\ep}(\sigma(t), \dot{\sigma}(t)) \: dt
\]
and taking the infimum in $\sigma$ we get $\varphi \leq c_{\xi}(x_0, .)$ on $\Omega$ so that $\varphi \leq \bar{c}_{\xi}(x_0, .)$ due to \ref{lem1}.
For the general case, if $\varphi$ is only $W^{1,p}(\Omega)$ and $\xi$ only $L_+^p(\theta)$, we first extend $\varphi$ to a function in $W^{1,p}(\mathbb{R}^d)$ and we extend $\xi$ outside $\Omega$ by writing $\xi(x,v) =  |\nabla \varphi| (x)$ for every $x \in \RR^d, v \in \Ss$ so that if $x \in \mathbb{R}^d \backslash \Ob$ and $u \in \Ss$ we have
\begin{align*}
 \nabla \varphi (x) \cdot u & \leq | \nabla \varphi(x) |   \\
& \leq  |\nabla \varphi(x)| |U|_1 = \Phi_{\xi}(x,u),
\end{align*}
where $U=(u_1,...,u_N) \in A_x^u$ is a minimizer of $ \Phi_{\xi}(x,u)$ The fact $U \in A_x^u$ implies that $|u|=1 \leq |U|_1$. By homogeneity of \eqref{eqlem} in $u$, \eqref{eqlem} thus continues to hold outside $\Omega$ with the previous extensions. 
We then regularize $\varphi$ and $\xi$. Let us take a mollifying sequence $\rho_n(x)=n^d \rho(nx)$, $x \in \mathbb{R}^d$ where $\rho$ is a smooth nonnegative function supported on the unit ball and such that $\int_{\mathbb{R}^d} \rho = 1$. Set $\xi^n := \rho_n \star \xi$ and $\varphi_n := \rho_n \star \varphi - (\rho_n \star \varphi ) (x_0). $ Let $x \in \mathbb{R}^d $. Recalling \ref{hy4} (H\"older condition on the $v_k$'s), we have 
\begin{align*}
\nabla \varphi_n (x) \cdot v_k(x) & =\int_{\RR^d} \rho_n(y) \nabla \varphi (x-y) \cdot v_k(x) \: dy \\
& \leq \int_{\RR^d} \rho_n(y) \xi_k(x-y) \: dy \\
& + \int_{\RR^d} \rho_n(y) \nabla \varphi (x-y) \cdot (v_k(x)- v_k(x-y)) \: dy\\
& \leq \xi_k^n (x) + n^{d- \alpha} \| \rho \|_{\infty} \int_{B(0,1/n)} | \nabla \varphi (x-y) | \: dy \\
& \leq \xi_k^n (x) + C n^{d- \alpha - d/q} \| \rho \|_{\infty} \| \nabla \varphi \|_p \\
& = \xi_k^n (x) + \ep_n,
\end{align*}
where $\ep_n >0,$ $ \ep_n \rightarrow 0$ as $n \rightarrow \infty$ (since $\alpha > d/p $).
So by using the above remark and the previous case where $\vp$ and $\xi$ were regular, we have $\varphi_n \leq c_{\xi^n + \ep_n} (x_0, .)$ and from the convergence of $\varphi_n$ to $\varphi$ it follows that
\[
\varphi = \limsup \varphi_n \leq \limsup c_{\xi^n + \ep_n} (x_0, \cdot) \leq \overline{c}_{\xi}(x_0, \cdot),
\]
where the last inequality is given by the definition of $\overline{c}_{\xi}$ as a supremum \eqref{cbx} and the relative compactness of $c_{\xi^n + \ep_n}$ in $C(\ObOb)$.
\end{proof}

We want to apply \ref{lem3} to $c(x_0, \cdot)$ so that we need $c(x_0, \cdot) \in W^{1,p}(\Omega)$ for every $x_0 \in \Omega$. Let $(e_1, \dots, e_d)$ given by \ref{hy8}, $\vp \in C_c^1(\Omega)$ and $x\Ep_0 \in \Oep$ such that $|x_0-x\Ep_0| \leq \ep$. Using the uniform convergence of $c\Ep(x\Ep_0,\cdot)$ to $c(x_0,\cdot)$ and \ref{hy6}, for $\vp \in C_c^1(\Omega)$ and $i=1,\dots, d$ we have 
\begin{multline*}
T_i \vp := \IO c(x_0, x) \nabla \vp (x) \cdot e_i (x) \, dx \\
 =  \lim_{\ep \rightarrow 0^+} \sum_{ \sigma \in C\Ep_i} \sum_{k=0}^{L(\sigma)-1} |y_{k+1} - y_k|^d c\Ep(x_0\Ep, y_k) \frac{\vp(y_{k+1}) - \vp(y_k)}{|y_{k+1} - y_k|} 
 \end{multline*}
 where $\sigma = (y_0, \dots,  y_{L(\sigma)})$.
Then we can rearrange the sums as follows 
\begin{align*}
T_i \vp & =  \\
& \lim_{\ep \rightarrow 0^+} \sum_{ \sigma \in C\Ep_i} \left ( \sum_{k=1}^{L(\sigma)-1} \vp(y_k) \left ( |y_k - y_{k-1} |^{d-1} c\Ep (x_0, y_{k-1}) - | y_{k+1} - y_k |^{d-1} c\Ep (x_0, y_k) \right ) \right. \\
& + \vp \left (y_{L(\sigma)} \right ) |y_{L(\sigma)} - y_{L(\sigma)-1} |^{d-1} c\Ep (x_0, y_{L(\sigma)-1})
- \vp(y_0) | y_1 - y_0| ^{d-1} c\Ep (x_0, y_0) \Bigg ) \\
& =  \lim_{\ep \rightarrow 0^+} \sum_{ \sigma \in C\Ep_i}  \sum_{k=1}^{L(\sigma)-1} \vp(y_k) \left ( |y_k - y_{k-1} |^{d-1} c\Ep (x_0, y_{k-1}) - | y_{k+1} - y_k |^{d-1} c\Ep (x_0, y_k) \right )
\end{align*}
since for $\ep$ small enough, $y_0$ and $y_{L(\sigma)}$ are not in the support of $\vp$ thanks to \ref{hy8}.
For $\sigma \in C\Ep_i$, we thus have
\begin{align*}
 \sum_{k=1}^{L(\sigma)-1} \vp(y_k) & \left ( |y_k - y_{k-1} |^{d-1} c\Ep (x_0, y_{k-1}) - | y_{k+1} - y_k |^{d-1} c\Ep (x_0, y_k) \right ) \\
& =  \sum_{k=1}^{L(\sigma)-1} \left ( \vp(y_k) [ |y_k - y_{k-1}|^{d-1} ( c\Ep(x_0, y_{k-1}) - c\Ep(x_0, y_k)) \right. \\
&  + \left. c\Ep(x_0, y_k) ( |y_k - y_{k-1}|^{d-1} - |y_{k+1} - y_k|^{d-1} ) ] \right ) \\
& \leq C \ep^d \sum_{k=1}^{L(\sigma)-1} |\vp(y_k)|. 
\end{align*}
Indeed in the first term, we use the fact that if $x$ and $y$ are neighbors in $\Oep$ then 
\[
c\Ep(x_0,x) \leq c\Ep(x_0,y) + |x-y| \max \xep (y, \cdot)
\]
and we obtain the upper bound on the second term due to \ref{hy8}. 
Hence by using H\"older and the fact that $\| \xep \|_{\ep,p}$ is bounded, we obtain 
\[
\left | \IO c(x_0, \cdot) \nabla \vp \cdot e_i \right | \leq C \| \vp \|_{L^{p'}}, \: \forall \vp \in C_c^1(\Omega).
\]
This proves that $c(x_0, \cdot) \in W^{1,p}(\Omega)$. By a similar argument we obtain that $c(\cdot ,y_0) \in W^{1,p}(\Omega)$ for every $y_0 \in \Omega$.

\begin{lemme} \label{lem4}
Let $x_0 \in \Omega$ and $c$ be defined by \eqref{4.8}, one has
\begin{enumerate}
\item For every $w \in C_c^{\infty} (\Omega, \mathbb{R}^d)$, the following inequality holds
\begin{equation} \label{eqlem4}
\int_{\Omega} \nabla_x c(x_0,x) \cdot  w(x) \: dx \leq \IO \Phi_{\xi}(x, w(x)) \: dx.
\end{equation}
\item $c \leq \cbx$ and so one has the $\Gamma$-liminf inequality. 
\end{enumerate}
\end{lemme}

\begin{proof} 

\textbf{1.)} Let $\alpha(x)=(\alpha_1(x),...,\alpha_N(x))$ be a minimizing decomposition of $w(x)$ i.e. for all $x \in \Omega$
\[
\inf_{X \in A_x^{w(x)}} \sum_{k=1}^N x_k \xi_k(x) = \sum_{k=1}^N \alpha_k(x) \xi_k(x) 
\]
with of course $w(x) = \sum \alpha_k(x) v_k(x)$ and $\alpha_k(x) \geq 0$.
Then we have 
\[
 \int_{\Omega} \nabla_x c(x_0,x) \cdot  w(x) \: dx = \sum_{k=1}^N \int_{\Omega} \alpha_k(x) \:  \nabla_x c(x_0,x) \cdot  v_k(x) \: dx.
\]
However the $\alpha_k$'s are not necessarily smooth so we must regularize the $\alpha_k$ to pass to the limit. As usual, we consider a mollifying sequence $(\rho^{\delta})$ (with $\delta > 0$), write 
\[
\alpha_k^{\delta} = \rho^{\delta} \star \alpha_k \: \text{ and } \: w^{\delta} = \sum_{k=1}^N \alpha_k^{\delta} v_k
\] 
for $k=1,...,N.$ Hence we have
\[
\int_{\Omega} \nabla_x c(x_0, \cdot ) \cdot  w = \lim_{\delta \rightarrow 0} \int_{\Omega} \nabla_x c(x_0, \cdot ) \cdot  w^{\delta}.
\]
Let $x_0^{\ep} \in \Nep$ such that $|x_0 - x_0^{\ep}| \leq \ep $ so that we have the uniform convergence of $c^{\ep} (x_0^{\ep}, \cdot)$ to $c(x_0, \cdot)$. By using \ref{hy6}, for every $\varphi \in C_c^1 (\Omega)$, we know that for $k=1,...,N$,
\begin{multline*}
 \int_{\Omega} c_k(x) \varphi(x) \nabla_x c(x_0, x ) \cdot  v_k(x) \, dx \\
 = \lim_{\ep \rightarrow 0^+}  \sum_{(x,e) \in \Eep_k} |e|^d \frac{c^{\ep}(x_0^{\ep},x+e) - c^{\ep}(x_0^{\ep},x)}{|e|} \varphi (x).
\end{multline*}
So we may write for a fixed $\delta$ 
\[
\int_{\Omega} \nabla_x c(x_0, \cdot ) \cdot  w^{\delta} =  \lim_{\ep \rightarrow 0^+} \sum_{k=1}^N \sum_{(x,e) \in \Eep_k} |e|^d \frac{c^{\ep}(x_0^{\ep},x+e) - c^{\ep}(x_0^{\ep},x)}{|e|} \frac{\alpha_k^{\delta}(x)}{c_k(x)}.
\]
Since $c^{\ep}(x_0^{\ep}, x+e) - c^{\ep}(x_0^{\ep},x) \leq |e| \xep (x, e) $, we obtain
\begin{align*} 
\int_{\Omega} \nabla_x c(x_0, \cdot ) \cdot  w^{\delta} & \leq  \lim_{\ep \rightarrow 0^+}  \sum_{(x,e) \in \Eep_k} |e|^d  \xep (x,e)  \frac{\alpha_k^{\delta}(x)}{c_k(x)} \\
& = \IO \alpha_k^{\delta} \xi_k.
\end{align*}

Passing to the limit in $\delta \rightarrow 0$, we finally get
\begin{align*}
\int_{\Omega} \nabla_x c(x_0, \cdot ) \cdot  w & \leq \sum_{k=1}^N \int_{\Omega} \alpha_k \: \xi_k \\
& = \int_{\Omega} \inf_{X \in A_x^{w(x)}} \left ( \sum_{k=1}^N x_k \xi(x, v_k(x)) \right ) \: dx. 
\end{align*}

\textbf{2.)} First, using \eqref{eqlem4} with $w = \theta v$ for $v \in C_c^{\infty} (\Omega, \RR^d)$ and an arbitrary scalar function $\theta \in C_c^{\infty}(\Omega, \RR), \theta \geq 0$, we deduce from the homogeneity of $z \mapsto \Phi_{\xi}(x,z)$ that
\beq \label{eqbislem4}
\nabla_x c(x_0,x) \cdot  v(x) \leq \Phi_{\xi}(x, v(x)), \text{ a.e. on } \Omega. 
\eeq
Now let $x$ be a Lebesgue point of both $\xi$ and $\nabla_x c(x_0, .)$, $u \in \Ss$ and take $v \in C_c^{\infty}(\Omega, \RR^d)$ such that $v=u$ in some neighbourhood of $x$. By integrating inequality \eqref{eqbislem4} over $B_r(x)$, dividing by its measure and letting $r \rightarrow 0^+$ we obtain  
\[
\nabla_x c(x_0,x) \cdot  u \leq \Phi_{\xi}(x, u), \text{ a.e. on } \Omega. 
\]
From \ref{lem3} the desired result follows.
\end{proof}

\subsection{The $\Gamma$-limsup inequality} 

Given $\xi \in L_+^p (\theta)$, we now prove the $\Gamma$-limsup inequality that is there exists a family $\xep \in \RR_+^{\# \Eep}$ such that
\beq \label{4.15}
\xep \rightarrow \xi \: \text{ and } \: \limsup_{\ep \rightarrow 0^+} J\Ep (\xep) \leq J(\xi).
\eeq
First show that for $\xi$ continuous and then a density argument will allow us to treat the general case. 

\textbf{Step 1 :} The case where $\xi$ is continuous. \\
For every $\ep > 0, (x,e) \in \Eep,$ write
\[
\xep (x, e) = \xi \left ( x, \frac{e}{|e|} \right).
\]
We have
\[
 \| \xep \|_{\ep, p} \rightarrow \|\xi\|_p \: \text{ and } \: I_0\Ep (\xep) \rightarrow I_0(\xi) \text{ as } \ep \rightarrow 0^+.
\]
In particular, for $\ep > 0$ small enough, $\| \xep \|_{\ep, p} \leq 2 \|\xi\|_p$ and $\xep \rightarrow \xi$ in the weak sense of \ref{def1}. We can proceed analogously to the construction \eqref{4.8} of $c$ for the $\Gamma$-liminf. We define $c\Ep$ on the whole of $\ObOb$ in a similar way and we also have the uniform convergence of $c\Ep$ to some $c$ in $C(\ObOb)$ (passing up to a subsequence) and $\liminf_{\ep \rightarrow 0^+} I_1\Ep(\xep) = \IObOb c d\gamma$ so that to prove \eqref{4.15} it is sufficient to show that $c \geq c_{\xi} = \cbx$. To justify this inequality it is enough to see that by construction for $(x,y) \in \Nep \times \Nep$ one has 
\[
c\Ep (x,y) = \inf_{\sigma \in C_{x,y}^{\ep}} \int_0^1 \Psi\Ep(\tilde{\sigma}(t), \dot{\tilde{\sigma}}(t)) \cdot \boldsymbol{\xep}(\tilde{\sigma}(t)) dt \geq c_{\xi}(x,y)
\]
using the uniform convergence of $c\Ep$ to $c$ we indeed obtain $c \geq c_{\xi} = \cbx$.

\textbf{Step 2 :} the general case where $\xi$ is only $L_+^p (\theta)$. \\
Let $\xi_n \in C(\ObS, \RR_+)$ such that
\[
\|\xi-\xi_n\|_p + \| c_{\xi_n} - \cbx \|_{\infty} + |I_0(\xi_n) - I_0(\xi) | \leq \frac{1}{n}
\]
and 
\[
\| \xi_n \|_p \leq 2 \| \xi \|_p
\]
(existence is given by \ref{convcxin}). For every $n > 0, \ep > 0$ there exists $\xep_n \in \RR_+^{\#\Eep}$ such that $\xep_n \rightarrow \xi_n$. Then there exists a nonincreasing sequence $\ep_n > 0$ converging to $0$ such that for every $0 < \ep < \ep_n$ we have
\[
|I_0\Ep(\xep_n) - I_0(\xi_n)| \leq \frac{1}{n}, \: I_1\Ep(\xep_n) \geq I_1(\xi_n) - \frac{1}{n} \: \text{ and } \: \| \xep_n \|_{\ep,p} \leq 2 \| \xi_n \|_p.
\] 
For $\ep > 0$, let $n_{\ep} = \sup \{ n ; \ep_n \geq \ep \}$ and $\xep = \xep_{n_{\ep}}$ then we get $\xep \rightarrow \xi$ ($\| \xep \|_{\ep, p} \leq 2 \| \xi_n \|_p \leq 4 \| \xi \|_p$) as well as 
\[
|I_0\Ep(\xep) - I_0(\xi)| \leq \frac{2}{n_{\ep}}  \rightarrow 0 \text{ as } \ep \rightarrow 0^+
\]
and
\[
I_1\Ep(\xep) \geq I_1(\xi_{n_{\ep}}) - \frac{1}{n_{\ep}} = \IOO c_{\xi_{n_{\ep}}} d \gamma - \frac{1}{n_{\ep}}.
\]
Since $c_{\xi_{n_{\ep}}}$ converges to $\cbx$, we then have
\[
\liminf I_1(\xep) \geq I_1(\ep)
\] 
which completes the proof.
 
 \section{Optimality conditions and continuous Wardrop equilibria} \label{sect5}
 
Now we are interested in finding optimality conditions for the limit problem:
 \begin{equation} \label{pc2}
 \inf_{\xi \in L_+^p(\theta)} J(\xi) := \IOS H(x, v, \xi(x, v)) \: \theta(dx, dv) - \int_{\overline{\Omega} \times \overline{\Omega}} \cbx \: d\gamma,
  \end{equation}
 through some dual formulation that can be seen in terms of continuous Wardrop equilibria. More precisely, it is in some sense the continuous version of the discrete minimization problem subject to the mass conservation conditions \eqref{cons1}-\eqref{cons2}.
Write
\[
\mathcal{L} = \{ (\sigma, \rho) : \sigma \in W^{1,\infty}([0,1], \Ob), \rho \in \mathcal{P}_{\sigma} \cap L^{\infty}([0,1])^N\},
\] 
  where
  \[ 
  \mathcal{P}_{\sigma} = \left \{ \rho : t \in [0,1] \mapsto \rho(t) \in \RR_+^N : \dot{\sigma}(t) = \sum_{k=1}^N v_k(\sigma(t)) \: \rho_k (t) \text{ a.e.} t \right \}.
  \]  
 We consider $\mathcal{L} $ as a subset of $C([0,1], \RR^d) \times L^1([0,1])^N$ i.e. equipped with the product topology, that on $C([0,1], \RR^d)$ being the uniform topology and that on $L^1([0,1])^N$ the weak topology. 
Slightly abusing notations, let us denote $\mathcal{M}_+^1(\mathcal{L})$ the set of Borel probability measures $Q$ on $C([0,1], \RR^d) \times L^1([0,1])^N$ such that $Q(\mathcal{L}) = 1$. For $\sigma \in W^{1,\infty}([0,1], \Ob)$, let us denote by $\tilde{\sigma}$ the constant speed reparameterization of $\sigma$ belonging to $W^{1,\infty}([0,1], \Ob)$ i.e. for $t \in [0,1], \tilde{\sigma}(t) = \sigma( s^{-1}(t)),$ where
\[
s(t) = \frac{1}{l(\sigma)}\int_0^t |\dot{\sigma}(u)| \: du \text{ with } l(\sigma) = \int_0^1 |\dot{\sigma}(u)| \: du.
\]
Likewise for $\rho \in \mathcal{P}_{\sigma} \cap L^{\infty}([0,1])^N$, let $\tilde{\rho}$ be the reparameterization of $\rho$ i.e. 
\[
\tilde{\rho}_k(t) := \frac{l(\sigma)}{| \dot{\sigma}(s^{-1}(t))|} \rho_k(s^{-1}(t)), \forall t \in [0,1], k=1, \ldots, N.
\]
We have $\tilde{\rho} \in \Pc_{\tilde{\sigma}} \cap L^{\infty}([0,1])^N$ with $\| \tilde{\rho} \|_{L_1} = \| \rho \|_{L^1}$. Define
\[
\Lct := \{ (\sigma, \rho) \in \Lc : |\dot{\sigma}| \text{ is constant} \} = \{ (\tilde{\sigma}, \tilde{\rho}), (\sigma, \rho) \in \Lc \}. 
\]
Let $Q \in \mathcal{M}_+^1(\Lc)$, we define $\widetilde{Q} \in \mathcal{M}_+^1(\Lct)$ as the push forward of $Q$ through the map $(\sigma, \rho) \rightarrow (\tilde{\sigma}, \tilde{\rho})$.  
Then let us define the set of probability measures on generalized curves that are consistent with the transport plan $\gamma$ :
\beq \label{Qg}
\mathcal{Q}(\gamma) := \{Q \in \mathcal{M}_+^1(\mathcal{L}) : (e_0, e_1)_{\#}Q= \gamma \},
\eeq 
where $e_0$ and $e_1$ are evaluations at time $0$ and $1$ and $(e_0, e_1)_{\#}Q$ is the image measure of $Q$ by $(e_0, e_1)$.
Thus $Q \in \mathcal{Q}(\gamma)$ means that
\[
\int_{\mathcal{L}} \varphi(\sigma(0), \sigma(1)) \: dQ(\sigma,\rho) := \int_{\overline{\Omega} \times \overline{\Omega}} \varphi(x,y) \: d\gamma(x,y), \: \: \forall \varphi \in C(\mathbb{R}^d \times \mathbb{R}^d, \mathbb{R}).
\]
This is the continuous analogue of the mass conservation condition \eqref{cons1} since $Q$ plays the same role as the paths-flows in the discrete model. Let us now write the analogue of the arc flows induced by $Q \in \mathcal{Q}(\gamma)$; for $k=1,\dots,N$ let us define the nonnegative measures on $\ObS$, $m_k^Q$ by
\[\int_{\ObS} \varphi(x,v) \: dm_k^Q(x,v) = \int_{\mathcal{L}} \left ( \int_0^1 \varphi(\sigma(t),v_k(\sigma(t))) \rho_k(t) dt \right ) dQ(\sigma, \rho),  \]
for every $\varphi \in C(\ObS, \mathbb{R}).$ 
Then the nonnegative measure on $\ObS$ $m^Q= \sum_{k=1}^N m_k^Q$ may be defined by 
\beq \label{mQ}
\int_{\ObS} \xi dm^Q = \int_{\mathcal{L}} L_{\xi}(\sigma, \rho) \: dQ(\sigma, \rho), \forall \xi \in C(\ObS, \RR_+)
\eeq
where for every $(\sigma, \rho) \in \Lc$, 
\beq \label{5.2}
L_{\xi}(\sigma, \rho) = \sum_{k=1}^N \int_0^1 \xi(\sigma(t),v_k(\sigma(t))) \rho_k(t) \: dt = \int_0^1 \boldsymbol{\xi} (\sigma(t)) \cdot \rho(t) \: dt,
\eeq 
with 
\[
\boldsymbol{\xi} (\sigma(t)) = (\xi(\sigma(t),v_1(\sigma(t))), \dots, \xi(\sigma(t),v_N(\sigma(t)))). 
\]

Notice that $\Lxsp = L_{\xi}(\tilde{\sigma}, \tilde{\rho})$ for every $(\sigma, \rho) \in \Lc$ and so $m^{\tilde{Q}} = m^Q$ for every $Q \in \mathcal{M}_+^1(\Lc)$.
The $p$ growth asumption \eqref{H3} on $H(x,v,\cdot)$ can be reformulated by a $q = p/(p-1)$ growth on $G(x,v,\cdot)$. To be more precise, we will assume that $g(x,v,\cdot)$ is continuous, positive and increasing in its last argument (so that $G(x,v,\cdot)$ is strictly convex) such that there exists $a$ and $b$ such that $0 < a \leq b$ and 
\begin{equation} \label{qcroi}
am^{q-1} \leq g(x,v,m) \leq b(m^{q-1}+1) \: \forall \, (x,v,m) \in \ObS \times \mathbb{R}_+, 
\end{equation}
with $q \in ( 1, d/(d-1) )$.
Then let us define
\beq \label{Qqg}
\mathcal{Q}^q(\gamma) := \{ Q \in \mathcal{Q}(\gamma) \: : \: m^Q \in L^q(\OS, \theta) \}
\eeq
and assume
\beq \label{Qqg}
\mathcal{Q}^q(\gamma) \neq \emptyset.
\eeq
This assumption is satisfied for instance when $\gamma$ is a discrete probability measure on $\ObOb$ and $q < d/(d-1)$. Indeed, first for $Q \in \Mc_+^1(W^{1,\infty}([0,1], \Ob))$, let us define $i_Q \in \Mc_+(\Ob)$ as follows 
\[
\IO \vp \: di_Q = \int_{W^{1,\infty}([0,1], \Ob)} \left ( \int_0^1 \vp(\sigma(t)) |\dot{\sigma}(t)| dt \right ) dQ(\sigma) \text{ for } \vp \in C(\Ob, \RR).
\]
It follows from \cite{benmansour2009numerical} that there exists $Q \in \Mc_+^1(W^{1,\infty}([0,1], \Ob))$ such that $(e_0, e_1)_{\#} Q = \gamma$ and $i_Q \in L^q$. For each curve $\sigma$, let $\rho^{\sigma} \in \Pc_{\sigma}$ such that $\sum_k \rho_k^{\sigma}(t) \leq C |\dot{\sigma}(t)|$ (we have the existence thanks to \ref{hy4}). Then we write $\Qb = {(id, \rho^{\cdot})}_{\#}Q$. We obtain $\Qb \in \mathcal{Q}^q(\gamma)$ so that we have proved the existence of such kind of measures. 

Let $Q \in \Qc^q(\gamma)$ and $\xi$ and $\tilde{\xi}$ be in $C(\ObS, \RR_+)$, we have
\begin{align*}
\IL \left |\Lxsp - L_{\tilde{\xi}} (\sigma, \rho) \right | d Q (\sigma, \rho) 
& = \IL \left | \int_0^1 ( \boldsymbol{\xi}(\sigma(t)) - \tilde{\boldsymbol{\xi}}(\sigma(t)) ) \cdot \rho(t) \: dt \right | dQ(\sigma, \rho) \\
& \leq \IOS \left |  \xi - \tilde{\xi} \right | m^Q \: \theta(dx, dv) \\
& \leq \| \xi - \tilde{\xi} \|_{L^p(\theta)} \| m^Q \|_{L^q(\theta)}.
\end{align*}
So if $\xi \in L_+^p(\theta)$ and $(\xi_n)_n$ is a sequence in $C(\ObS, \RR_+)$ that converges in $L^p(\theta)$ to $\xi$ then $L_{\xi_n}$ is a Cauchy sequence in $L^1(\Lc, Q)$ and its limit (that we continue to denote by $L_{\xi}$) does not depend on the approximating sequence $(\xi_n)_n$. This suggests us to define $L_{\xi}$ in an $L^1(\Lc, Q)$ sense for every $\xi \in L_+^p(\theta)$ and $Q \in \mathcal{Q}^q(\gamma)$. For every $\xi \in L_+^p(\theta)$ and $Q \in \mathcal{Q}^q(\gamma)$, by proceeding as for Lemma $3.6$ in \cite{carlier2008optimal}, we have
\beq \label{5.6}
\IOS \xi \cdot m^Q \: \theta(dx, dv) = \IL \Lxsp \: d Q(\sigma, \rho),
\eeq
and
\beq \label{5.6b} 
\cbx (\sigma(0), \sigma(1)) \leq L_{\xi}(\sigma, \rho) \: \text{ for } Q-\text{a.e. } (\sigma, \rho) \in \Lc.
\eeq
Hence using the fact that $Q \in \mathcal{Q}^q (\gamma)$ and \eqref{5.6}-\eqref{5.6b}, we obtain
\beq \label{5.8}
\IObOb \cbx d\gamma = \IL \cbx(\sigma(0), \sigma(1)) \: dQ(\sigma, \rho) \leq \IOS \xi \cdot m^Q.
\eeq
Let $\xi \in L_+^p(\theta)$ and $Q \in \mathcal{Q}^q(\gamma)$, it follows from Young's inequality that 
\begin{multline} \label{5.7}
\IOS H(x, v, \xi(x,v)) \: \theta(dx, dv) \\
 \geq \IOS \xi \cdot m^Q \: \theta(dx, dv) - \IOS G \left (x, v, m^Q(x,v) \right ) \: \theta(dx, dv)
\end{multline}    
so that we have
\beq \label{5.9}
\inf_{\xi \in L_+^p(\theta)} J(\xi) \geq \sup_{Q \in \mathcal{Q}^q(\gamma)} - \int_{\OS} G \left (x, v, m^Q(x,v) \right ) \: \theta (dx, dv). 
\eeq
The dual formulation of \eqref{pc2} then is
\begin{equation} \label{pc3}
\sup_{Q \in \mathcal{Q}^q(\gamma)} - \int_{\OS} G \left (x, v, m^Q(x,v) \right ) \: \theta (dx, dv). 
\end{equation}
We can note the analogy between \eqref{pc3} and the discrete problem that consists in minimizing \eqref{P1} subject to the mass conservation conditions \eqref{cons1}-\eqref{cons2}.
Then we establish the following theorem, that specifies relations between \eqref{pc3} and \eqref{pc2} and that  gives the connection with Wardrop equilibria:
\begin{theo} \label{th2}
\begin{enumerate}
Under assumptions \eqref{qcroi} and \eqref{Qqg}, we have:
\item \eqref{pc3} admits solutions.
\item $\Qb \in \mathcal{Q}^q(\gamma)$ solves \eqref{pc3} if and only if 
\beq \label{eqQ}
\int_{\mathcal{L}} L_{\xi_{\overline{Q}}}(\sigma, \rho) \: d \overline{Q} (\sigma, \rho) = \int_{\mathcal{L}} \overline{c}_{\xi_{\overline{Q}}}(\sigma(0), \sigma(1) ) \: d \overline{Q} (\sigma, \rho)
\eeq
where $\xi_{\overline{Q}}(x, v) = g \left (x,v, m^{\overline{Q}}(x,v) \right ) $. 
\item Equality holds : $\inf \eqref{pc2} = \sup \eqref{pc3}$. Moreover if $\overline{Q}$ solves \eqref{pc3} then $\xi_{\overline{Q}}$ solves \eqref{pc2}.
\end{enumerate}
\end{theo}

It is the main result of this section. To prove it, we need some lemmas. First, let us start with a preliminary lemma on the $v_k$'s that is a consequence of \ref{hy4}. 

\begin{lemme} \label{remhy4}
For all subset $I \subset \{1, \dots, N \}$, it is all or nothing, that is, we are in one of the two following cases : 
\begin{enumerate} 
\item $0 \in$ Conv$\left ( \{v_i(x) \}_{i \in I} \right )$ for every $x \in \Ob$, 
\item $0 \notin$ Conv$\left ( \{v_i(x) \}_{i \in I} \right )$ for every $x \in \Ob$. 
\end{enumerate}
Moreover, there exists a constant $0 < \delta < 1$ such that for all subset $I \subset \{1, \dots, N \}$ that is in the second case, there exists $u_x \in \text{Conv}(\{v_i(x)\}_{i \in I}$ for all $x \in \Ob$ such that
\[
v_i(x) \cdot \frac{u_x}{|u_x|} \geq \delta \text{ for all  } i \in I.
\]
\end{lemme}  

\bpr
We will use the fact that $\Ob$ is connected. The first property is obviously closed since the $v_k$'s are continuous. Let us now show that the second one is closed. 
Let $I \subset \{1, \dots, N \}$, assume by contradiction that there exists a sequence $\{x_n\}_{n \geq 0} \in \Ob^{\NN}$ converging to $x \in \Ob$ such that $0 \notin C_n=$ Conv$\left ( \{v_i(x_n) \}_{i \in I} \right )$ for every $n \geq 0$ and $0 \in C=$ Conv$\left ( \{v_i(x) \}_{i \in I} \right )$. So there exists $\{ \lambda_i \}_{i \in I}$ such that $\sum_{i \in I} \lambda_i v_i(x) =0, \lambda_i \geq 0$ and $\sum_{i \in I} \lambda_i = 1$. Without loss of generality, we can assume that the $\lambda_i$'s are positive. Then we have that $v^n = \sum_{i \in I} \lambda_i v_i(x_n) \neq 0$ and converges to $0$ as $n \rightarrow +\infty$. Let $\beta_n > 0$ such that $|\beta_n v^n | = 1$ then $\beta_n$ converges to $+\infty$. Thanks to \ref{hy4}, with $\xi = (\xi_1, \dots, \xi_N), \xi_i = 0$ if $i \in I, 1$ otherwise, there exists $\{ z_i^n \}_{i \in I} \in (\RR_+^{\# I})^{\NN}$ such that $|z_i^n| \leq C$ and $\sum_{i\in I} z_i^n v_i(x_n) = \beta_n v_i(x_n)$ for all $i \in I$ and $n \geq 0$. Then we obtain  that $\sum_{i \in I} (\beta_n \lambda_i - z_i^n) v_i(x_n) = 0$. But for n large enough, we have that $\beta_n \lambda_i - z_i^n > 0$ for every $i \in I$, which is a contradiction. 

For a subset $I \subset \{1, \dots, N \}$ that is in the second case, Conv$(\{v_i(x)\}_{i \in I})$ is contained in a salient (pointed) cone for all $x \in \Ob$. For all $x \in \Ob$, we can think of $u_x$ as being in the medial axis of this cone. Since the $v_k$'s are continuous and $\Ob$ is compact, we have the desired result.
\epr 

Now let us define the following sets 
\[
\Lc^{C} = \left \{ (\sigma, \rho) \in \Lc : \sum_{k=1}^N \rho_k(t) \leq C |\dot{\sigma}(t)| \text{ a.e. } t \in [0,1] \right \}
\]
for some constants $C > 0$. 
 Now let us notice that we can simplify the problem \eqref{pc3} with the following lemma: 

\begin{lemme} \label{pbred}
For a well-chosen constant $C'> 1$, one has  
\begin{multline*} \label{pdcontrole}
\inf_{Q \in \mathcal{Q}^q(\gamma)} \int_{\OS} G \left (x,v, m^Q(x,v) \right ) \: \theta (dx, dv) \\
= \inf_{Q \in \mathcal{Q}^q(\gamma)} \left \{ \int_{\OS} G \left (x,v, m^Q(x,v) \right ) \: \theta (dx, dv) : Q \left ( \Lc^{C'} \right ) = 1\right \}.
\end{multline*}
\end{lemme}
 
 \bpr
  We set $C'=1/\delta$ where $\delta$ is given by \ref{remhy4}. Let $(\sigma, \rho) \in \Lc$. We will prove that there exists $\rhob \in \Pc_{\sigma}$ such that for all $t \in [0,1]$, $\rhob_k(t) \leq \rho_k(t)$ for all $k=1, \dots, N$ and $\sum_{k=1}^N \rhob_k(t) \leq C' |\dot{\sigma}(t)|$. 
Let $t \in [0,1]$ such that $\sum_{k=1}^N \rho_k(t) > C' |\dot{\sigma}(t)|$. Let us denote $I$ the subset of $\{1, \dots, N \}$ such that for every $k \in I, \rho_k(t) > 0$. First, if $0 \in$ Cone$(\{v_k(\sigma(t))\}_{k \in I})$, there exists a conical combination of $0$
\[ 
\sum_{k \in I} \lambda_k v_k(\sigma(t)) = 0 
\]
with the $\lambda_k$'s $\geq 0$. Then we write $\rhob_k(t) = \rho_k(t) - \lambda \lambda_k$ (we take $\lambda_k = 0$ for $k \notin I$) where 
\[
\lambda = \min_{k \in I : \lambda_k \neq 0} \left \{ \frac{\rho_k(t)}{\lambda_k} \right \}.
\]
We set $\bar{I}$ the subset of $I$ such that for every $k \in \bar{I}, \rhob_k(t) > 0$. We restart with $\rhob$ and we continue until $0 \notin$ Conv$(\{v_k(\sigma(t))\}_{k \in \bar{I}})$. Let $u$ be as in \ref{remhy4} for $I=\bar{I}$ and $x = \sigma(t)$. Then we have 
\[
|\dot{\sigma}(t) | \geq \dot{\sigma}(t) \cdot u = \sum_{k=1}^N \rhob_k(t) v_k(\sigma(t)) \cdot u  \geq \delta \sum_{k=1}^N \rhob_k(t) 
\]    
so that $\sum_{k=1}^N \rhob_k(t) \leq C' |\dot{\sigma}(t)|$. For $Q \in \Mc_+^1(\Lc)$, we denote by $\Qb \in \Mc_+^1(\Lc^{C'})$ the push forward of $Q$ through the map $(\sigma, \rho) \mapsto (\sigma, \rhob)$. Then we have $m^{\Qb} \leq m^Q$. Since $G(m,v, \cdot)$ is nondecreasing, we have : 
\[
\int_{\OS} G \left (x,v, m^{\Qb}(x,v) \right ) \: \theta (dx, dv) \leq \int_{\OS} G \left (x,v, m^Q(x,v) \right ) \: \theta (dx, dv).
\] 
\epr

To prove that the problem \eqref{pc3} has solutions, a natural idea would be to take a maximizing sequence $\{Q_n\}_{n \geq 0}$ for \eqref{pc3} and to show that it converges to $Q \in \Qc^q(\gamma)$ that solves \eqref{pc3}. For this, we would like to use Prokhorov's theorem which would allow us to obtain the tightness of $\{\Qt_n\}$ and $\star$-weak convergence of $\{Q_n\}$ to a measure in $\Mc_+^1(\Lc)$. Unfortunately, the space $C([0, 1], \RR^d) \times L^1([0, 1])^N$ is not a Polish space for the considered topology (because of the weak topology of $L^1([0,1])$). So we will work with Young's measures in order to apply Prokhorov's theorem. Let us define the set $C = C([0,1], \RR^d) \times \mathfrak{P}_1(\RR^d \times [0,1])$ where for a Polish space $(E,d)$, we set 
\[
\mathfrak{P}_1(E) = \left \{ \mu \in \Mc_+^1(E) : \int_E \text{d}(x,x') \: d\mu(x) < +\infty \: \text{ for some } x' \in E \right \}.  
\]
We equip $C$ with the product topology, that on $C([0,1], \RR^d)$ being the uniform topology and $\Pf_1(\RR^d \times [0,1])$ being endowed with the $1$-Wasserstein distance 
\[
W_1(\mu^1, \mu^2) := \min \left \{ \int_{E^2} \text{d}(x_1, x_2) \: d \mu(x_1, x_2) : \mu \in \Pi(\mu^1, \mu^2) \right \} 
\]
where $E = \RR^d \times [0,1]$, d is the usual distance on $E$, $(\mu^1, \mu^2) \in \Pf_1(E)^2$ and $\Pi(\mu^1, \mu^2)$ is the set of transport plans between $\mu^1$ and $\mu^2$, that is, the set of probability measures $\mu$ on $E^2$, having $\mu^1$ and $\mu^2$ as marginals: 
\beq \label{defplantransp}
\int_{E \times E} \vp(x) d \mu (x,y) = \int_E \vp(x) d\mu^1(x) \text{ and } \int_{E \times E} \vp(y) d \mu (x,y) = \int_E \vp(x) d\mu^2(x), 
\eeq
for every $\vp \in C(E, \RR)$. The set $C$ is a Polish space (see \cite{ambrosio2008gradient}). Let us denote by $\lambda$ the Lebesgue measure on $[0,1]$ and let us consider the subset $\Sc$ of $C$ : 
\[
\Sc = \left \{ (\sigma, \nu_t \otimes \lambda) : \sigma \in W^{1,\infty}([0,1], \Ob), \nu_t \otimes \lambda \in \Pf_1(E), \nu_t \in \Mfs^t \text{ a.e. } t  \right \},
\]
where for $t \in [0,1]$ and $\sigma \in W^{1,\infty}([0,1], \Ob)$, 
\[
\Mfs^t = \left \{ \nu_t \in \Mc_+^1(\RR^d) : \text{supp } \nu_t \subset \bigcup_{k=1}^N \RR_+ v_k(\sigma(t)) \text{ and } \dot{\sigma}(t) = \int_{\RR^d} v \: d\nu_t(v) \right \}. 
\]

The Young measures $\nutl$ are the analogue of the decompositions $\rho \in \Pc_{\sigma}$. For the general theory of the Young measures, see for instance \cite{pedregal2012parametrized}.

Let us define the set of probability measures on curves $(\sigma, \nu_t \otimes \lambda)$ that are consistent with the transport plan $\gamma$ :
\beq \label{Xg}
\Xc(\gamma) := \{X \in \mathcal{M}_+^1(\Sc) : (e_0, e_1)_{\#}X= \gamma \}.
\eeq 
This is the analogue of \eqref{Qg}. Let us now write the analogue of $m^Q$ (given by \eqref{mQ}) as follows:
\beq \label{iX}
\int_{\ObS} \xi di^X = \int_{\Sc} \Lb_{\xi}(\sigma, \kappa) \: dX(\sigma, \kappa), \forall \xi \in C(\ObS, \RR)
\eeq
where for every $(\sigma, \kappa = \nu_t \otimes \lambda) \in \Lc$, 
\beq \label{5.2b}
\Lb_{\xi}(\sigma, \kappa) = \int_0^1 \left ( \int_{\RR^d} \xi \left (\sigma(t),\frac{v}{|v|} \right) |v| \: d\nu_t(v) \right ) dt. 
\eeq 
Then let us define
\beq \label{Xqg}
\Xc^q(\gamma) := \{ X \in \Xc(\gamma) \: : \: i^X \in L^q(\OS, \theta) \}
\eeq
 
Let $X \in \Xc^q(\gamma)$. By the same reasoning as for $Q \in \Qc^q(\gamma)$, if $\xi \in L_+^p(\theta)$, we denote by $\Lb_{\xi}$ the limit of the Cauchy sequence $\Lb_{\xi_n}$ in $L^1(\Sc, X)$ for any sequence $(\xi_n)_n$ converging in $L^p(\theta)$ to $\xi$. 
We may write the analogue of the problem \eqref{pc3} :
 \begin{equation} \label{pcX}
\sup_{X \in \Xc^q(\gamma)} - \int_{\OS} G \left (x,v, i^X(x,v) \right ) \: \theta (dx, dv). 
\end{equation}

\begin{lemme} \label{egpb}
One has $\sup \eqref{pc3} =  \sup \eqref{pcX}$.
\end{lemme}

\bpr 
Let $Q \in \Qc^q(\gamma)$ and $\sigma \in W^{1,\infty}([0,1], \Ob)$. For $ \rho = (\rho_1, \dots, \rho_N) \in \Pc_{\sigma}$, we define the  Young's measure $\nu^{\rho}_t \otimes \lambda$ as follows :
\[
\nu^{\rho}_t = \sum_{k=1}^N \frac{\rho_k(t)}{|\rho(t)|_1} \delta_{ \{|\rho(t)|_1 v_k(\sigma(t))\} },
\]
where $|\rho(t)|_1 = \sum_{k=1}^N \rho_k(t)$ for every $t \in [0,1]$.
We consider the measure $X^Q \in \Xc^q(\gamma)$ defined by
\[
\int_{\Sc} \vp \: dX^Q = \IL \vp(\sigma, \nu^{\rho}_t \otimes \lambda ) \: dQ(\sigma, \rho), \text{ for all } \vp \in C(\Sc, \RR).
\]
Since we have $m^Q = i^{X^Q}$ we immediately get $\sup \eqref{pc3} \leq \sup  \eqref{pcX}$. 

For the converse inequality, let $X \in \Xc^q(\gamma)$, we build $Q^X \in \Qc^q(\gamma)$. Let $(\sigma, \nu_t \otimes \lambda) \in \Sc$, recalling that one has $\text{supp } \nu_t \subset \bigcup_{k=1}^N \RR_+ v_k(\sigma(t))$ for $t \in [0,1]$, we define $\rho^{\nu} \in \Pc_{\sigma}$ as follows 
\[
\rho^{\nu}_k(t) = \int_{\RR_+ v_k(\sigma(t))} |v| \: d\nu_t(v), \text{ for all } k=1, \dots, N
\]
and $\rho^{\nu} = (\rho^{\nu}_1, \dots, \rho^{\nu}_N)$ if the $v_k(\sigma(t))$'s are pairwise distinct. Otherwise, let us decompose $\{ 1, \dots, N\} = \bigcup_{j=1}^s I_j$ where the $I_k$'s are pairwise disjoint and such that for all $j=1,\dots,s$ and $k \in I_j, v_k(\sigma(t)) = v_j$ where the $v_j$'s are pairwise distinct. Then for all $j=1,\dots,s$ and $k \in I_j$, we set
\[
\rho^{\nu}_k(t) = \frac{1}{\# I_j} \int_{\RR_+ v_j} |v| \: d\nu_t(v). 
\]  
The element $\rho^{\nu}$ is in $\Pc_{\sigma}$. Similarly, we set  
\[
\int_{\Lc} \vp \: dQ^X = \ISc \vp(\sigma, \rho^{\nu} ) \: dQ(\sigma, \nutl) \text{ for all } \vp \in C(\Lc, \RR).
\]
From the fact that $m^{Q^X} = i^X$ it follows that $\sup \eqref{pc3} \geq \sup \eqref{pcX}$. 
\epr

Let us notice that with the previous proof for $(\sigma, \nu_t \otimes \lambda) \in \Sc$, we may build $\nut$ as a sum of Dirac measures : 
\[
\nut_t = \sum_{k=1}^N \frac{\rho^{\nu}_k (t)}{|\rho^{\nu}(t)|_1} \delta_{\{|\rho^{\nu}(t)|_1 v_k(\sigma(t))\}} 
\]  
where $\rho^{\nu}$ is given in the previous proof. Therefore it follows from the same reasoning as in the proof of \ref{pbred} that we may take $\rhonu \in \Pc_{\sigma}$ such that for $t \in [0,1], \sum_{k=1}^N \rho_k(t) \leq C' |\dot{\sigma}(t)|$. Moreover, we can choose $(\sigma, \rho)$ only in $\Lct$ with $|\dot{\sigma}|$ constant. Then the new measure $\sum_{k=1}^N \frac{\rhonu_k (t)}{|\rhonu(t)|_1} \delta_{\{|\rho(t)|_1 v_k(\sigma(t))\}}$ that we continue to denote by $\nut_t$ by abuse of notations is in $\Mfs^t$. Let us define
\begin{multline*}
\Sc^{C'} = \{(\sigma, \nu_t \otimes \lambda) \in \Sc : \text{supp } \nu_t \cap \RR_+ v_k(\sigma(t)) = \{ \rho(t) v_k(\sigma(t)) \} \\
\text{ with } \rho(t) \leq C' |\dot{\sigma}(t)| \text{ for } k=1, \dots, N \text{ and } t \in [0,1] \}
\end{multline*} 
and 
\[
\Sct = \{ (\sigma, \nu_t \otimes \lambda) \in \Sc : |\dot{\sigma}| \text{ is constant} \}.
\]
For $X \in \Mc_+^1(\Sc)$, we denote by $\Xt \in \Mc_+^1(\Sc^{C'} \cap \Sct)$ the push forward of $X$ through the map $(\sigma, \nu_t \otimes \lambda) \mapsto (\tilde{\sigma}, \nut_t \otimes \lambda)$. Then we have $i^{\Xt} \leq i^X$. Since $G(m,v, \cdot)$ is nondecreasing, we may consider only the measures $\Xt \in \Mc_+^1(\Sc^{C'} \cap \Sct)$ for the problem \eqref{pcX}. 

We now adapt the proof in \cite{baillon2012discrete}. In particular we have to generalize Lemmas $2.7$ and $2.8$ in \cite{carlier2008optimal}, this becomes
\begin{lemme} \label{lem2.7}
For every $\varphi \in C(\ObS, \RR_+)$, $\Lb_{\varphi}$ is l.s.c. on $\Sc$ for the topology defined above.
\end{lemme}

\begin{proof}
Let $(\sigma,\nu_t \otimes \lambda) \in \Sc$ and $(\sigma^n, \nu_t^n \otimes \lambda)$ be a sequence converging to $(\sigma,\nu_t \otimes \lambda) \in \Sc$. Then by definition, we have 
\[
\Lb_{\vp}(\sigma^n,\nu_t^n \otimes \lambda) = \int_0^1 \left ( \int_{\RR^d} \vp \left ( \sigma^n(t),\frac{v}{|v|} \right ) |v| \: d\nu_t^n(v) \right ) dt. 
\]
$\sigma^n \rightarrow \sigma$ in $C([0,1])$ so that $\vp(\sigma^n(\cdot),\frac{v}{|v|})$ converges strongly in $L^{\infty}$. Since $\nu_t^n$ narrowly converges to $\nu_t$ and the function $(t,v) \mapsto \vp(\sigma(t),\frac{v}{|v|}) |v|$ is the upper limit of $(t,v) \mapsto \vp(\sigma(t),\frac{v}{|v|}) \min (|v|, n)$, that is continuous and bounded, as $n \rightarrow +\infty$, we obtain the desired result. 
\end{proof}

\begin{lemme} \label{lem2.8}
Let $(X_n)_n \in \mathcal{M}_+^1(C)^{\NN}$ be such that $X_n(\Sc^{C'}) = 1$ for every $n$ and there exists a constant $M>0$ such that 
\[
\sup_n \int_{\Sc} l(\sigma) \: dX_n(\sigma, \nu_t \otimes \lambda) \leq M.
\]
Then the sequence $(\Xt_n)_n$ is tight and admits a subsequence that weakly-$\star$ converges to a probability measure $X$ such that $X(\Sc) = 1$. 
\end{lemme}

\bpr
For every $K >0$, let us define the following subset of $\Sct^{C'}$ 
\[
\Sct_K = \left \{ (\sigma, \nutl) \in \Sct  : |\dot{\sigma}| \leq K \text{ and supp } \nu_t \leq B_{C' K} \right \}
\]
where $C'$ is the constant given by \ref{pbred}.
Let us show that $\Sct_K$ is relatively compact in $\Sc$. First, the set $\{ \sigma \in W^{1, \infty}([0,1], \Ob) : \sigma \: K\text{-Lipschitz continuous} \}$ is compact in $C([0,1], \Ob)$ thanks to Ascoli's theorem. The set of probability measures with support in $B_{C'K}$ is compact due to the Banach-Alaoglu-Bourbaki theorem. Let a sequence $(\sigma^n, \nu_t^n \otimes \lambda) \in (\Sct_K)^{\NN}$ converging to $(\sigma, \nutl) \in \Sc$, prove that $(\sigma, \nutl) \in \Sct_K$. 

%

\textbf{1.)} $\text{supp } \nu_t \subset \bigcup_{k=1}^N \RR_+ v_k(\sigma(t)). $

First let us notice that the function $\vp : (x,v) \mapsto \text{dist}(v, \bigcup_{k=1}^N \RR_+ v_k(x))$ is continuous on $\RR^d \times \Ob$. 
We then have : 
\begin{align*}
\int_0^1 \left ( \int_{\RR^d} \vp(v,\sigma(t)) \: d \nu_t(v) \right ) dt & = \int_0^1 \left ( \int_{B_{C' K}} \vp(v,\sigma(t)) \: d \nu_t(v) \right ) dt \\
& = \lim_{n \rightarrow +\infty} \int_0^1 \left ( \int_{B_{C' K}} \vp(v,\sigma^n(t)) \: d \nu^n_t(v) \right ) dt \\
& = 0.
\end{align*}
So $\vp(x,v) = 0$ $d \nu_t \otimes dt$-a.e. and the support of $\nu_t$ is in $\bigcup_{k=1}^N \RR_+ v_k(\sigma(t))$ for $t \in [0,1]$.     


\textbf{2.)} $\dot{\sigma}(t) = \int_{\RR^d} v \: d\nu_t(v)$.

By definition, for $n \geq 0$ and $(s,t) \in [0,1]^2$, we have : 
\[
\sigma^n(t) - \sigma^n(s) = \int_s^t \int_{B_{C'K}} v \: d\nu_t^n(v) \otimes \lambda = (v \mathbf{1}_{B_{C'K}} \otimes \mathbf{1}_{[s,t]} ; \nu_t^n \otimes \lambda). 
\]
Obviously, the sequence $\{ \sigma^n(t) - \sigma^n(s) \}_{n \geq 0}$ converges to $\sigma(t) - \sigma(s)$ (since $\sigma^n$ uniformly converges to $\sigma$). For the term in the right-hand side, it is sufficient to take a sequence $\{\Phi_{\ep}\}_{\ep >0}$ in $C_b(\RR^d \times [0,1])^{\NN}$ converging to $(v,t) \mapsto v \mathbf{1}_{B_{C'K}} \otimes \mathbf{1}_{[s,t]}$ in $L^1(\RR^d \times [0,1])$ as $\ep \rightarrow 0^+$. 

Now let us justify the tightness of $(\Xt_n)_n$:
\begin{align*}
\Xt_n \left ( (\Sct_K)^c \right ) & \leq \Xt_n \left ( \left \{ (\sigma, \nutl) \in \Sct \cap \Sc^{C'} : |\dot{\sigma}| > K \right \} \right ) \\
& + \Xt_n \left ( \left \{ (\sigma, \nutl) \in \Sct \cap \Sc^{C'} : \text{supp} (\nu_t) \nsubseteq B_{C' K} \right \} \right ) \\
& \leq 2 \Xt_n \left ( \left \{ (\sigma, \nutl) \in \Sct \cap \Sc^{C'} : |\dot{\sigma}| > K \right \} \right ) \\
& \leq 2 X_n \left ( \left \{ (\sigma, \nutl) \in \Sc : l(\sigma) > K \right \} \right ) \\
& \leq \frac{2}{K} \int_{\Sc} l(\sigma) \: dX_n(\sigma, \nu_t \otimes \lambda) \\
& \leq 2 \frac{M}{K} \rightarrow 0 \text{ as } K \rightarrow + \infty. 
\end{align*}

Due to Prokhorov's theorem we can then assume that passing up to a subsequence, $(\Xt_n)_n$ weakly-$\star$ converges to $X \in \mathcal{M}_+^1(C)$. It remains to show that $X(\Sc) = 1$. For $K > 0$, let us define the closed set
\[
\Sc_K = \left \{ (\sigma, \nutl) \in \Sc  : l(\sigma) \leq K \text{ and supp } \nu_t \subset B_{C' K} \right \}.
\]
It follows from the previous computation, the fact that the measures $\Xt_n$ are concentrated on $\Sct$ and Portmanteau's theorem that
\begin{align*}
1 = \limsup_n  \Xt_n(\Sc) & \leq \limsup_n \Xt_n(\Sc_K) + \limsup_n \Xt_n (\Sc \backslash \Sc_K) \\
& \leq X (\Sc_K) + \frac{M}{K}.
\end{align*}
Letting $K$ tend to $\infty$, we then obtain $X(\Sc) = \sup_K X(\Sc_K) = 1$. 
\epr

\begin{lemme} \label{lem2.9}
Let $(X_n)_n$ be a sequence in $\mathcal{M}_+^1(\Sc)$ that weakly star converges to some $X \in \mathcal{M}_+^1(\Sc)$. If there exists $i \in \mathcal{M}_+(\ObS)$ such that $i^{X_n}$ weakly-$\star$ converges to $i$ in $\mathcal{M}_+(\ObS)$ then we have $i^X \leq i$.  
\end{lemme}

The proof is similar to that of Lemma $2.9$ in \cite{carlier2008optimal}.


%

\begin{proof} (of \ref{th2}) 

Let us prove the existence of solutions for the problem \eqref{pc3}. Thanks to \ref{egpb}, we consider the problem \eqref{pcX}. Due to \eqref{qcroi} the value of problem \eqref{pcX} is finite. Let $(X_n)_n$ be a maximizing sequence of \eqref{pcX}. Since $i^X \geq i^{\Xt}$, we can assume $X_n = \Xt_n$ for all $n$. Still from \eqref{qcroi} it follows that $(i^{X_n})_n$ is bounded in $L^q(\theta)$. So, passing up to a subsequence, we can assume that $(i^{X_n})_n$ weakly converges in $L^q(\theta)$ to some $i$. Moreover, since  $(i^{X_n})_n$ is bounded in $L^q(\theta)$ so in $L^1(\theta)$, we have
\begin{align*}
\sup_n \ISc l(\sigma) \: dX_n(\sigma, \nutl) & \leq \sup_n \ISc \left ( \int_0^1 \left ( \int_{\RR^d} |v| d\nu_t(v) \right ) dt \right ) dX_n(\sigma, \nutl) \\
& = \sup_n \IOS d i^{X_n} < + \infty. 
\end{align*}
Since $X_n = \Xt_n$, we can deduce from \ref{lem2.8} that, up to a subsequence, $(X_n)_n$ weakly-$\star$ converges to some $X \in \mathcal{M}_+^1(\Sc)$. Using the fact that $\mathcal{X}(\gamma)$ is weakly closed, we see that $X \in \mathcal{X}(\gamma)$ and \ref{lem2.9} then imply that $i^X \leq i$ so that $X \in \mathcal{X}^q(\gamma)$. Since $G(x, v, \cdot)$ is convex and nondecreasing, we then have 
\begin{align*}
\IOS G(x, v, i^X(x, v) \: \theta(dx, dv) & \leq \IOS G(x, v, i(x, v)) \: \theta(dx, dv) \\
& \leq \liminf_n \IOS G(x, v, i^{X_n}(x, v)) \: \theta(dx, dv),  
\end{align*}
which proves that $X$ solves \eqref{pcX}. Thus as mentioned in the proof of \ref{egpb}, there exists $Q \in \Qc^q(\gamma)$ such that $m^Q = i^X$ and so $Q$ is a solution of \eqref{pc3}. 

The reasoning for the last two statements is similar to that of Theorem $5.1$ in \cite{baillon2012discrete}.
\epr

A natural question is to investigate the discrete problems corresponding to \eqref{P1} i.e.
\begin{equation} \label{pd4}
\inf_{\boldsymbol{m\Ep}, \boldsymbol{w\Ep}} \sum_{(x,e) \in \Eep} |e|^d G \left (x, \frac{e}{|e|}, \frac{m\Ep(x, e)}{|e|^{d/2}} \right )
\end{equation}
subject to the mass conservation conditions \eqref{cons1}-\eqref{cons2} and convergence of problems \eqref{pd4} in some sense to the continuous problem
\begin{equation} \label{pc4}
\inf_{Q \in \mathcal{Q}(\gamma)} \IOS G (x, v, m^Q(x, v)) \: \theta(dx, dv).
\end{equation}

Let $\mathbf{m\Ep} = (m\Ep(x,e))_{(x,e) \in \Eep}$ and $\mathbf{w\Ep} = (w\Ep(\sigma))_{\sigma \in C\Ep}$ solve the discrete problem \eqref{pd4}. Let $\sigma = (x_0, \dots, x_{L(\sigma)})\in C\Ep$ (identified with the piecewise affine curve defined on $[0, L(\sigma)]$). For every $k = 0, \dots, L(\sigma)-1,$ let us denote by $i_k$ the integer such that $(x_k, x_{k+1} - x_k) \in \Eep_{i_k}$. Then let us define $\rho^{\sigma} \in L^{\infty}([0,1])^N$ where for all $t \in [k,  k+1[$, 
\[
\rho_i^{\sigma}(t) = 
\left \{ 
\begin{aligned} & | \sigma(k+1) - \sigma(k) | & \text{ if } i = i_k, \\
& 0  & \text{ otherwise.}  
\end{aligned}
\right.
\]
We will define a discrete measure $Q\Ep$ over $\mathcal{L}\Ep$ where
 \[
\mathcal{L}\Ep = \{ (\sigma, \rho^{\sigma}) : \sigma \in C\Ep\}.
\] 
Write $Q\Ep$ as follows
\[
Q\Ep := \ep^{d/2-1} \sum_{\sigma \in C\Ep} w\Ep (\sigma) \delta_{\sigma \otimes \rho^{\sigma}}
\]
as well as
\[
\Qt\Ep := \ep^{d/2-1} \sum_{\sigma \in C\Ep} w\Ep (\sigma) \delta_{\tilde{\sigma} \otimes \rho^{\tilde{\sigma}}}
\]
where $\tilde{\sigma} \in W^{1,\infty}([0,1], \Ob)$ is the constant speed reparameterization of the path $\sigma$. Notice that for every $\xi \in C(\ObS, \RR_+)$, we have $L_{\xi}(\sigma, \rho^{\sigma}) = L_{\xi}(\tilde{\sigma}, \rho^{\tilde{\sigma}})$ so that $m^{Q\Ep} = m^{\Qt\Ep}$. Let us also observe that the measure $m^{\Qt\Ep}$ contains all the information on $(\mathbf{m\Ep}, \mathbf{w\Ep})$. 

Especially for the following theorem, we make a stronger assumption.
\begin{hyp} \label{hy9}
There exists a function $C : \RR_+ \mapsto \RR_+^*$ such that $C(\ep) \rightarrow 1$ as $\ep \rightarrow 0^+$ and for every $\ep >0, (x,e) \in \Eep, C(\ep) \ep \leq |e| \leq \ep$.
\end{hyp}
In particular, this hypothesis is satisfied in our three classical examples since arc length is constant for $\ep > 0$ fixed.  

\begin{theo}
Under the previous assumptions, defining $\Qt\Ep$ as above, up to a subsequence, $(\Qt\Ep)_{\ep} > 0$ weakly converges to some solution $Q \in \mathcal{Q}^q(\gamma)$ of \eqref{pc4} in the sense that  
\[
\int_{C([0,1], \RR^d) \times L^1([0,1])^N} \Phi (\sigma, \rho) d \Qt\Ep(\sigma, \rho) \rightarrow \int_{C([0,1], \RR^d) \times L^1([0,1])^N} \Phi (\sigma, \rho) d Q(\sigma, \rho),
\]
as $\ep \rightarrow 0^+$ for every $\Phi \in C_b(C([0,1], \RR^d) \times L^1([0,1])^N, \RR)$.
\end{theo}

\bpr
By duality, from \ref{th2} and \ref{coro1}, it follows that the value of \eqref{pd4} converges to that of \eqref{pc4} and in particular, due to the $q$ growth condition \eqref{qcroi} on $G(x, v, \cdot)$, $\mathbf{m\Ep}$ is bounded for the discrete $L^q$ norm. In the same manner that in the proof of \ref{coro1} and Section $4.1$ we can see that there is some $m \in L_+^q$ such that $(x,e) \rightarrow \frac{m\Ep(x, e)}{|e|^{d/2}}$ weakly converges to $m$ in $L^q$ in the sense of \ref{def1} (up to replacing $p$ by $q$) and 
\beq \label{5.14}
\IOS G(x, v, m(x,v)) \: \theta (dx, dv) \leq \liminf_{\ep \rightarrow 0^+} \sum_{(x, e) \in \Eep} |e|^d G \left ( x, \frac{e}{|e|}, \frac{m\Ep(x,e)}{|e|^{d/2}} \right ). 
\eeq
Let $\xi \in C(\ObS, \RR_+)$, recalling \eqref{cons2}, \eqref{H5} and \eqref{mQ}, rearranging terms, we have
\begin{align*}
& \IOS \xi(x,v) \: d m^{\Qt\Ep}(x,v)  = \int_{\mathcal{L}} L_{\xi}(\sigma, \rho) \: d \Qt\Ep(\sigma, \rho) \\
& = \ep^{d/2-1} \sum_{\sigma \in C\Ep} w\Ep (\sigma) \sum_{k=0}^{L(\sigma)-1} \int_k^{k+1} \xi ( \sigma(t), v_{i_k}(\sigma(t))) |\sigma(k+1)-\sigma(k)| dt \\
& = \ep^{d/2-1} \sum_{\sigma \in C\Ep} w\Ep (\sigma) \sum_{k=0}^{L(\sigma)-1} \int_{[\sigma(k), \sigma(k+1)]} \xi \left ( \cdot, \frac{\sigma(k+1)-\sigma(k)}{|\sigma(k+1)-\sigma(k)|} \right ) + O(w_{\xi}(\ep)) \\
& = \ep^{d/2-1} \sum_{\sigma \in C\Ep} w\Ep (\sigma) \sum_{k=0}^{L(\sigma)-1} \left ( \xi \left ( \sigma(k), \frac{\sigma(k+1)-\sigma(k)}{|\sigma(k+1)-\sigma(k)|} \right ) + O(w_{\xi}(\ep)) \right) |\sigma(k+1)-\sigma(k)| \\
& = \ep^{d/2-1} \sum_{(x, e) \in \Eep} \left ( \xi \left ( x, \frac{e}{|e|} \right ) + O(w_{\xi}(\ep)) \right )  \left ( \sum_{\sigma \in C\Ep : [x, x+e] \subset \sigma} |e| w\Ep (\sigma) \right ) \\
& = \ep^{d/2-1} \sum_{(x, e) \in \Eep} |e|^{d/2+1} \xxe \frac{m\Ep(x, e)}{|e|^{d/2}} + O(w_{\xi}(\ep))
\end{align*}
where $w_{\xi}$ is a modulus of continuity of $\xi$.  From \ref{hy9} and the fact that $(x,e) \rightarrow \frac{m\Ep(x, e)}{|e|^{d/2}}$ weakly converges in $L^q$ to $m$ in the sense of \ref{def1}, it follows that $m^{\Qt\Ep}$ weakly star converges to $m$. Arguing as previously, we find $Q \in \mathcal{M}_1^+(\mathcal{L})$ such that, up to a subsequence, $(\Qt\Ep)_{\ep}$ weakly converges to $Q$ and $m^Q \leq m$. We easily have $Q \in \mathcal{Q}^q(\gamma)$ : indeed, for every $\vp \in C(\RR^d \times \RR^d, \RR)$, we have 
\begin{align*}
\IL \vp(\sigma(0), \sigma(1)) \: dQ(\sigma, \rho) & = \lim_{\ep \rightarrow 0^+} \IL \vp(\sigma(0), \sigma(1)) \: d\tilde{Q}\Ep (\sigma, \rho) \\
& =  \lim_{\ep \rightarrow 0^+} \ep^{d/2-1} \sum_{\sigma \in C\Ep} w\Ep(\sigma) \vp(\sigma(0), \sigma(1)) \\
& =  \lim_{\ep \rightarrow 0^+} \ep^{d/2-1} \sum_{(x,y) \in {N\Ep}^2} \vp(x,y) \left ( \sum_{\sigma \in C_{x,y}\Ep} w\Ep(\sigma) \right ) \\
& =  \lim_{\ep \rightarrow 0^+} \ep^{d/2-1} \sum_{(x,y) \in {N\Ep}^2} \vp(x,y) \gamma\Ep(x,y) \\
& = \IObOb \vp \: d\gamma. 
 \end{align*}
Using \eqref{5.14} and the fact that $G(x, v, \cdot)$ is nondecreasing, we get
\begin{align*}
\IOS G(x, v, m^Q(x, v)) \: \theta(dx, dv) & \leq \IOS G(x, v, m(x, v)) \: \theta(dx, dv) \\
& \leq \liminf_{\ep \rightarrow 0^+} \sum_{(x, e) \in \Eep} |e|^d G \left (x, \frac{e}{|e|}, \frac{m\Ep(x,e)}{|e|^{d/2}} \right ).
\end{align*}
Since the right-hand side is the value of the infimum in \eqref{pc4}, we obtain the desired result.
\epr

\section{The long-term variant}

Instead of taking the transport plan $\gamma\Ep$ as given in the discrete problem, we now consider the case where  only its marginals are fixed. More precisely, there is a distribution of sources $f_-\Ep= \sum_{x \in N\Ep} f_-\Ep(x) \delta_x$ and sinks $f_+\Ep= \sum_{x \in N\Ep} f_+\Ep(x) \delta_x$ which are discrete measures with same total mass on the set of nodes $\Nep$ (that we can assume to be 1 as a normalization)
\[
\sum_{x \in N\Ep} f_-\Ep(x) = \sum_{y \in N\Ep} f_+\Ep(y) = 1.
\]
The numbers $f_-\Ep(x)$ and $f_+\Ep(x)$ are nonnegative for every $x \in \Nep$. 

With the same notations as in the short-term problem, we have almost the same definition of an equilibrium as in \ref{defW}, we must change the mass conservation condition \eqref{cons1} as follows
\beq \label{cons1bis}
f\Ep_-(x) := \sum_{\sigma \in C_{x, \cdot}\Ep} w\Ep(\sigma), \:  f\Ep_+(y) := \sum_{\sigma \in C_{\cdot, y}\Ep} w\Ep(\sigma)
\eeq
for every $(x,y) \in N\Ep \times N\Ep$, where $C_{x,\cdot}\Ep$ (respectively $C_{\cdot,y}\Ep$) is the set of loop-free paths starting at the origin $x$ (respectively stopping at the terminal point $y$). Moreover, the transport plan now is an unknown. Similar arguments apply to this case, the equilibrium is a minimizer of the functional defined by \eqref{P1} but now subject to \eqref{cons1bis} and \eqref{cons2}. We shall then state the analogue of the dual formulation \eqref{D1} 
\begin{equation} \label{D1bis}
 \inf_{t^{\ep} \in \mathbb{R}_+^{\#E^{\ep}}} \left \{ \sum_{(x,e) \in E^{\ep}} H^{\ep} (x, e, t^{\ep} (x,e)) - \inf_{\gamma\Ep \in \Pi(f_-\Ep, f_+\Ep)}\sum_{(x,y) \in {N^{\ep}}^2} \gamma^{\ep}(x,y) T_{t^{\ep}}^{\ep}(x,y) \right \},
 \end{equation}
where $\Pi(f_-\Ep, f_+\Ep)$ is the set of discrete transport plans between $f_-\Ep$ and $f_+\Ep$, that is, the set of nonnegative numbers $(\gamma\Ep(x,y))_{(x,y) \in {N\Ep}^2}$ such that
\[
\sum_{y \in N\Ep} \gamma\Ep(x,y) = f_-\Ep(x), \: \sum_{x \in N\Ep} \gamma\Ep(x,y) = f_+\Ep(y), \: \forall (x,y) \in N\Ep \times N\Ep,
\]
We assume that the hypotheses made in Subsection $3.1$ are still satisfied, except that we replace \ref{hy1} by 
\begin{hyp} \label{hy1bis}
$f_-\Ep$ and $f_+\Ep$ weakly star converge to some probability measures $f_-$ and $f_+$ on $\Ob$:
\[
\lim_{\ep \rightarrow 0^+} \ep^{d/2-1}\sum_{x \in N\Ep} ( \vp(x) f_-\Ep(x) + \psi(x) f_+\Ep(x)) = \IOb \vp df_- + \IOb \psi df_+, \: \forall (\vp, \psi) \in C(\Ob)^2. 
\]
\end{hyp}
Writing $\xep$ as in \eqref{xep}, we can now reformulate \eqref{D1bis}  
\beq  \label{pd1bis}
 \inf_{\xep \in \mathbb{R}_+^{\#E^{\ep}}} F^{\ep}(\xep) := I_0^{\ep}(\xep) - F_1^{\ep}(\xep)
\eeq
where $I_0^{\ep}(\xep)$ is defined by \eqref{pd1.1} and 
\begin{equation} \label{pd1.2b}
F_1^{\ep}(\xep) := \inf_{\gamma\Ep \in \Pi(f_-\Ep, f_+\Ep)}\sum_{(x,y) \in {N^{\ep}}^2} \gamma^{\ep}(x,y) \left ( \min_{\sigma \in C_{x,y}^{\ep}} \sum_{(z,e) \subset \sigma} |e|^{d/2} \xep(z,e) \right ) .
\end{equation} 
It is an optimal transport problem. The limit functional then reads as the following variant of \eqref{pc1}
\beq \label{plb}
F(\xi) := I_0(\xi) - F_1(\xi), \text{ where } F_1(\xi) := \inf_{\gamma \in \Pi(f_-, f_+)} \IObOb \cbx d\gamma, \: \forall \xi \in L_+^p,
\eeq
As previously, $I_0$ is defined by \eqref{I0} and $\cbx$ by \eqref{cbx}. $\Pi(f_-, f_+)$ is the set of transport plans between $f_-$ and $f_+$ (see \eqref{defplantransp}). We then have the following $\Gamma$-convergence result :
\begin{theo}
Under the same assumptions except \ref{hy1} replaced by \ref{hy1bis}, the family of functionals $F\Ep$ defined by \eqref{pd1bis} $\Gamma$-converges (for the weak $L^p$-topology) to the functional $F$ defined by \eqref{plb}.
\end{theo}

Same arguments as for Theorem $6.1$ in \cite{baillon2012discrete} apply here. 


In the same manner as in \ref{sect5}, we can see that the problem \eqref{plb} has a dual formulation that is
\begin{equation} \label{pdb}
\sup_{Q \in \mathcal{Q}^q(f_-, f_+)} - \IOS G (x, v, m^Q(x, v)) \theta(dx, dv),
\end{equation}
where 
\begin{align*}
\mathcal{Q}^q(f_-, f_+) & := \{ Q \in \mathcal{M}_1^+(\Lc) : {e_0}_{\#} Q = f_-, {e_1}_{\#}Q = f_+, m^Q \in L^q(\theta) \} \\
& = \bigcup_{\gamma \in \Pi(f_-, f_+)} \mathcal{Q}^q(\gamma).
\end{align*}
If we assume that $\mathcal{Q}^q(f_-, f_+) \neq \emptyset$ and that \eqref{qcroi} is still true, one can reformulate \ref{th2} for the long-term models as follows :
\begin{theo} \label{th2b}
\begin{enumerate}
We have : 
\item Problem \eqref{pdb} admits solutions,
\item $\overline{Q} \in \mathcal{Q}^q(f_-, f_+)$ solves \eqref{pdb} if and only if 
\[
\int_{\mathcal{L}} L_{\xi_{\overline{Q}}}(\sigma, \rho) \: d \overline{Q} (\sigma, \rho) = \int_{\mathcal{L}} \overline{c}_{\xi_{\overline{Q}}}(\sigma(0), \sigma(1) ) \: d \overline{Q} (\sigma, \rho)
\]
where $\xi_{\overline{Q}}(x, v) = g(x,v, m^{\overline{Q}}(x,v)) $ and moreover, $\overline{\gamma} := (e_0, e_1)_{\#}\overline{Q}$ is a solution of the optimal transport problem: 
\[
\inf_{\gamma \in \Pi(f_-, f_+)} \IObOb \overline{c}_{\xi_{\overline{Q}}}(x,y) d\gamma(x, y). 
\]
\item There is no duality gap : the infimum of \eqref{plb} equals the supremum of \eqref{pdb} and moreover, if $\overline{Q}$ solves \eqref{pdb} then $\xi_{\overline{Q}}$ solves \eqref{plb}.
\end{enumerate}
\end{theo}

Problem \eqref{pdb} is studied in \cite{hatchi2015wardrop}. It is showed that problem \eqref{pdb} is equivalent to another problem that is the variational formulation of an anisotropic, degenerate and elliptic PDE :
\[
\left \{
\begin{aligned}
 -\text{ div } (\nabla  \mathcal{G}^*(x, \nabla u(x))) & = f \: & \text{ in } \Omega, \\
   \nabla  \mathcal{G}^*(x, \nabla u(x)) \cdot \nu_{\Omega} & = 0 \: & \text{ on } \partial \Omega, 
\end{aligned} \right.
\]
with $\Gc^*$ being a $C^1$ function. In particular, if the function $g$ in \eqref{H2} is of the form $g(x, v_k(x), m) = a_k(x) m^{q-1} + \delta_k$ for every $x \in \Ob, k=1, \dots, N$ and $m \geq 0$ where the constants $\delta_k$ are positive and the weights $a_k$ are regular and positive, then we have
\[
 \mathcal{G}^*(x, z) = \sum_{k=1}^N \frac{b_k(x)}{p} (z \cdot v_k(x) - \delta_k c_k(x) )_+^p \text{ for every } x \in \Ob, z \in \RR^d
\]
where $b_k = (a_k c_k)^{-\frac{1}{q-1}}$. This case is interesting since numerical simulations can be performed as shown in \cite{hatchi2015wardrop}. 

\textbf{Acknowledgements} The author would like to thank Guillaume Carlier for his extensive help and advice.

\bibliographystyle{plain}
\bibliography{generalref}

 
 \end{document}